\documentclass[final,3p]{elsarticle}
\usepackage[latin1]{inputenc}
 \usepackage{graphics}
 \usepackage{graphicx}
 \usepackage{epsfig}
\usepackage{amssymb}
 \usepackage{amsthm}
 \usepackage{lineno}
 \usepackage{amsmath}
   \numberwithin{equation}{section}
\usepackage{mathrsfs}

\newtheorem{thm}{Theorem}[section]

\newtheorem{lem}[thm]{Lemma}
\newtheorem{prop}[thm]{Proposition}
\newtheorem{defn}[thm]{Definition}

\biboptions{sort&compress,square}
\allowdisplaybreaks
\begin{document}
\begin{frontmatter}
\author{Tong Wu$^1$}
\ead{wut977@nenu.edu.cn}
\author{Yong Wang$^2$\corref{cor2}}
\ead{wangy581@nenu.edu.cn}
\cortext[cor2]{Corresponding author.}

\address{1.School of Mathematics and Statistics, Northeast Normal University,
Changchun, 130024, China}
\address{2.School of Mathematics and Statistics, Northeast Normal University,
Changchun, 130024, China}

\title{ Gauss-Bonnet theorems associated to deformed Schouten-Van Kampen connection in the affine group and the group of rigid motions of the Minkowski plane }
\begin{abstract}
In this paper, we define deformed Schouten-Van Kampen connections which are metric connections and compute sub-Riemannian limits of Gaussian curvature for a Euclidean $C^2$-smooth surface associated to deformed Schouten-Van Kampen connections with two kinds of distributions in the affine group and the group of rigid motions of the Minkowski plane away from
 characteristic points and signed geodesic curvature
 for Euclidean $C^2$-smooth curves on surfaces. According to above results, we get Gauss-Bonnet theorems associated to two kinds of deformed Schouten-Van Kampen connections in the affine group and the group of rigid motions of the Minkowski plane.
\end{abstract}
\end{frontmatter}
\section{Introduction}
\indent In \cite{DV}, Diniz and Veloso gave the definition of Gaussian curvature for non-horizontal surfaces in sub-Riemannian Heisenberg space $\mathbb{H}^1$ and the proof of the Gauss-Bonnet theorem. In \cite{BTV}, intrinsic
Gaussian curvature for a Euclidean $C^2$-smooth surface in the Heisenberg group $\mathbb{H}^1$ away
from characteristic points and intrinsic signed geodesic curvature for Euclidean
$C^2$-smooth curves on surfaces are defined by using a Riemannian approximation scheme. These results were then used to prove a Heisenberg version of
the Gauss-Bonnet theorem. In \cite{Ve}, Veloso verified that Gausssian curvature of surfaces and normal curvature of curves in surfaces introduced by \cite{DV} and by \cite{BTV} to prove Gauss-Bonnet theorems
in Heisenberg space $\mathbb{H}^1$ were unequal and he applied the same formalism of \cite{DV} to
get the curvatures of \cite{BTV}. With the obtained formulas, the Gauss-Bonnet theorem can be proved as a straightforward application of Stokes theorem in \cite{Ve}.\\
\indent In \cite{BTV} and \cite{BTV1}, Balogh-Tyson-Vecchi used the Riemannian approximation scheme which
can depend upon the choice of the complement to the horizontal distribution in general. In \cite{BTV}, they proposed an interesting question to understand to what extent similar
phenomena hold in other sub-Riemannian geometric structures. In \cite{YS}, Wang and Wei gave sub-Riemannian limits of Gaussian curvature for a Euclidean $C^2$-smooth surface in the affine group and the group of rigid motions of the Minkowski plane away from characteristic points and signed geodesic curvature for Euclidean $C^2$-smooth curves on surfaces. And they got Gauss-Bonnet theorems in the affine group and the group of rigid motions of the Minkowski plane. In \cite{YS1}, Wang and Wei gave sub-Riemannian limits of Gaussian curvature for a Euclidean $C^2$-smooth surface in the BCV spaces and the twisted Heisenberg group away from characteristic points and signed geodesic curvature for Euclidean $C^2$-smooth curves on surfaces. And they got Gauss-Bonnet theorems in the BCV spaces and the twisted Heisenberg group. In \cite{BK}, Klatt proved a Gauss-Bonnet theorem associated to a metric connection (see Proposition 5.2 in \cite{BK}). In \cite{Ba} and \cite{zo}, Schouten-Van Kampen connections on foliations and almost (para) contact manifolds were studied.\\
\indent In this paper, we define deformed Schouten-Van Kampen connections which are metric connections and compute sub-Riemannian limits of Gaussian curvature for a Euclidean $C^2$-smooth surface associated to deformed Schouten-Van Kampen connection with two kinds of distributions in the affine group and the group of rigid motions of the Minkowski plane away from
 characteristic points and signed geodesic curvature
 for Euclidean $C^2$-smooth curves on surfaces. According to above results, we get Gauss-Bonnet theorems associated to two kinds of deformed Schouten-Van Kampen connections in the affine group and the group of rigid motions of the Minkowski plane.

\indent In Section 2, we prove Gauss-Bonnet theorems associated to the first kind of deformed Schouten-Van Kampen connection in the affine group.  In Section 3, we prove Gauss-Bonnet theorems associated to the second kind of deformed Schouten-Van Kampen connection in the affine group.
In Section 4, we prove Gauss-Bonnet theorems associated to the first kind of deformed Schouten-Van Kampen connection in the group of rigid motions of the Minkowski plane. In Section 5, we prove Gauss-Bonnet theorems associated to the second kind of deformed Schouten-Van Kampen connection in the group of rigid motions of the Minkowski plane.
\section{Gauss-Bonnet theorems associated to the first kind of deformed Schouten-Van Kampen connection in the affine group}
\indent Firstly we recall the affine group. Let $\mathbb{M}$ be the affine group which satisfies $$(m,n,s)\star(\lambda,\mu,\nu)=(m\lambda,m\mu+n,\nu+s),$$
where $(1,0,0)$ is the unit element of $\mathbb{M}$.\\
Let
\begin{equation}
X_1=x_1\partial_{x_1}, ~~X_2=x_1\partial_{x_2}+\partial_{x_3},~~X_3=x_1\partial_{x_2},
\end{equation}
with brackets
\begin{equation}
[X_1,X_2]=X_3,~~[X_1,X_3]=X_2,~~[X_2,X_3]=0.
\end{equation}
Then
\begin{equation}
\partial_{x_1}=\frac{1}{x_1}X_1, ~~\partial_{x_2}=\frac{1}{x_1}X_3,~~\partial_{x_3}=X_2-X_3,
\end{equation}
and ${\rm span}\{X_1,X_2,X_3\}=T\mathbb{M}.$
Let $\omega_1=\frac{1}{x_1}dx_1,~~\omega_2=dx_3,~~\omega=\frac{1}{x_1}dx_2-dx_3.$ For the constant $L>0$, let
$g_L=\omega_1\otimes \omega_1+\omega_2\otimes \omega_2+L\omega\otimes \omega$ be the Riemannian metric on $\mathbb{M}$. Then $X_1,X_2,\widetilde{X_3}:=L^{-\frac{1}{2}}X_3$ are orthonormal basis on $T\mathbb{M}$ with respect to $g_L$.\\
\indent Let $H_1={\rm span}\{X_1,X_2\}$ be the first kind of horizontal distribution on $\mathbb{M}$, then $H_1^\bot={\rm span}\{X_3\}$. Let $\nabla$ be the Levi-Civita connection on $\mathbb{M}$ with respect to $g_L$, and we recall the Schouten-Van Kampen connection $\nabla^{1,\alpha,s}$ by the following formulas
\begin{equation}
\nabla^{1,\alpha,s}_XY=P^1\nabla_X{P^1Y}+P^{1,\bot}\nabla_X{P^{1,\bot} Y},
\end{equation}
where $P^1$(resp.$P^{1,\bot}$) be the projection on $H_1$ (resp.$H_1^\bot$).\\
Nextly, we define the first kind of deformed Schouten-Van Kampen connection which is a metric connection in the affine group:
\begin{align}
\nabla^{1,\alpha}_XY&=(1-\alpha)\nabla_XY+\alpha\nabla^{1,\alpha,s}_XY\nonumber\\
&=(1-\alpha)\nabla_XY+\alpha P^1\nabla_X{P^1Y}+\alpha P^{1,\bot}\nabla_X{P^{1,\bot} Y},\nonumber\\
\end{align}
where $\alpha$ is a constant.\\
 By lemma 2.1 in \cite{YS} and (2.5), we have the following lemma
\begin{lem}
Let $\mathbb{M}$ be the affine group, then
\begin{align}
&\nabla^{1,\alpha}_{X_1}X_1=0,~~~\nabla^{1,\alpha}_{X_1}X_2=\frac{1-\alpha}{2}X_3,~~~ \nabla^{1,\alpha}_{X_1}X_3=-\frac{(1-\alpha)L}{2}X_2,\nonumber\\
&\nabla^{1,\alpha}_{X_2}X_1=-\frac{1-\alpha}{2}X_3,~~~\nabla^{1,\alpha}_{X_2}X_2=0,~~~\nabla^{1,\alpha}_{X_2}X_3=\frac{(1-\alpha)L}{2}X_1,\nonumber\\
&\nabla^{1,\alpha}_{X_3}X_1=-\frac{L}{2}X_2-(1-\alpha)X_3,~~~\nabla^{1,\alpha}_{X_3}X_2=\frac{L}{2}X_1,~~\nabla^{1,\alpha}_{X_3}X_3=(1-\alpha)LX_1.\nonumber\\
\end{align}
\end{lem}
\begin{defn}(\cite{YS})
Let $\gamma:[a,b]\rightarrow (\mathbb{M},g_L)$ be a Euclidean $C^1$-smooth curve. We say that $\gamma$ is regular if $\dot{\gamma}\neq 0$ for every $t\in [a,b].$ Moreover we say that
$\gamma(t)$ is a horizontal point of $\gamma$ if
$$\omega(\dot{\gamma}(t))=\frac{\dot{\gamma}_2(t)}{\gamma_1(t)}-\dot{\gamma}_3(t)=0.$$
\end{defn}
Similar to Definition 2.3 in (\cite{YS}), we have
\begin{defn}
Let $\gamma:[a,b]\rightarrow (\mathbb{M},g_L)$ be a Euclidean $C^2$-smooth regular curve in the Riemannian manifold $(\mathbb{M},g_L)$. The curvature $k^{L,\nabla^{1,\alpha}}_{\gamma}$ of $\gamma$ at $\gamma(t)$ is defined as
\begin{equation}
k^{L,\nabla^{1,\alpha}}_{\gamma}:=\sqrt{\frac{||\nabla^{1,\alpha}_{\dot{\gamma}}{\dot{\gamma}}||_L^2}{||\dot{\gamma}||^4_L}-\frac{\langle \nabla^{1,\alpha}_{\dot{\gamma}}{\dot{\gamma}},\dot{\gamma}\rangle^2_L}{||\dot{\gamma}||^6_L}}.
\end{equation}
\end{defn}
Then, we have
\begin{lem}
Let $\gamma:[a,b]\rightarrow (\mathbb{M},g_L)$ be a Euclidean $C^2$-smooth regular curve in the Riemannian manifold $(\mathbb{M},g_L)$. Then,
\begin{align}
k^{L,\nabla^{1,\alpha}}_{\gamma}&=\Bigg\{\Bigg\{\left[\frac{\ddot{\gamma}_1\gamma_1-(\dot{\gamma}_1)^2}{\gamma_1^2}+L\omega(\dot{\gamma}(t))\left(\frac{(1-\alpha)\dot{\gamma_2}(t)}{\gamma_1}+\frac{\alpha\dot{\gamma_3}(t)}{2}\right)\right]^2+\left[\ddot{\gamma}_3-\frac{(2-\alpha)\dot{\gamma}_1L}{2\gamma_1}\omega(\dot{\gamma}(t)\right]^2\nonumber\\
&+L\left[\frac{d}{dt}\omega(\dot{\gamma}(t))-\frac{(1-\alpha)\dot{\gamma}_1}{\gamma_1}\omega(\dot{\gamma}(t))\right]^2\Bigg\}
\cdot\left[\left(\frac{\dot{\gamma}_1}{\gamma_1}\right)^2+\dot{\gamma}_3^2+L(\omega(\dot{\gamma}(t)))^2\right]^{-2}\nonumber\\
&-\Bigg\{\frac{\dot{\gamma}_1}{\gamma_1}\left[\frac{\dot{\gamma}_1\ddot{\gamma}_1-(\dot{\gamma}_1)^2}{\gamma_1^2}+L\omega(\dot{\gamma}(t))\left(\frac{(1-\alpha)\dot{\gamma_2}(t)}{\gamma_1}+\frac{\alpha\dot{\gamma}_3(t)}{2}\right)\right]+\dot{\gamma}_3(t)\left[\ddot{\gamma}_3-\frac{(2-\alpha)\dot{\gamma}_1L}{2\gamma_1}\omega(\dot{\gamma}(t)\right]\nonumber\\
&+L\omega(\dot{\gamma}(t))\left[\frac{d}{dt}\omega(\dot{\gamma}(t))-\frac{(1-\alpha)\dot{\gamma}_1}{\gamma_1}\omega(\dot{\gamma}(t))\right]\Bigg\}^2\cdot\left[\left(\frac{\dot{\gamma}_1}{\gamma_1}\right)^2+\dot{\gamma}_3^2+L(\omega(\dot{\gamma}(t)))^2\right]^{-3}\Bigg\}^{\frac{1}{2}}.\nonumber\\
\end{align}
When $\omega(\dot{\gamma}(t))=0$, we have\\
\begin{align}
k^{L,\nabla^{1,\alpha}}_{\gamma}&=\Bigg\{\Bigg\{\left[\frac{\ddot{\gamma}_1\gamma_1-(\dot{\gamma}_1)^2}{\gamma_1^2}\right]^2+\ddot{\gamma}_3^2+L\left[\frac{d}{dt}\omega(\dot{\gamma}(t))\right]^2\Bigg\}
\cdot\left[\left(\frac{\dot{\gamma}_1}{\gamma_1}\right)^2+\dot{\gamma}_3^2\right]^{-2}\nonumber\\
&-\Bigg\{\frac{\dot{\gamma}_1}{\gamma_1}\left[\frac{\dot{\gamma}_1\ddot{\gamma}_1-(\dot{\gamma}_1)^2}{\gamma_1^2}+\dot{\gamma}_3\ddot{\gamma}_3\right]\Bigg\}^2\cdot\left[\left(\frac{\dot{\gamma}_1}{\gamma_1}\right)^2+\dot{\gamma}_3^2\right]^{-3}\Bigg\}^{\frac{1}{2}}.\nonumber\\
\end{align}
\end{lem}
\begin{proof}
By (2.3), we have
\begin{equation}
\dot{\gamma}(t)=\frac{\dot{\gamma}_1}{\gamma_1}X_1+\dot{\gamma}_3X_2+\omega(\dot{\gamma}(t))X_3.
\end{equation}
By Lemma 2.1 and (2.10), we have
\begin{align}
&\nabla^{1,\alpha}_{\dot{\gamma}}X_1=-\frac{L}{2}\omega(\dot{\gamma}(t))X_2-(1-\alpha)\left(\frac{\dot{\gamma_3}(t)}{2}+\omega(\dot{\gamma}(t))\right)X_3,\nonumber\\
&\nabla^{1,\alpha}_{\dot{\gamma}}X_2=\frac{L}{2}\omega(\dot{\gamma}(t))X_1+\frac{(1-\alpha)\dot{\gamma_1}(t)}{2\gamma_1(t)}X_3,\nonumber\\
&\nabla^{1,\alpha}_{\dot{\gamma}}X_3=(1-\alpha)L\left(\frac{\dot{\gamma_3}(t)}{2}+\omega(\dot{\gamma}(t))\right)X_1-(1-\alpha)\frac{L\dot{\gamma}_1(t)}{2\gamma_1(t)}X_2.\nonumber\\
\end{align}
By (2.10) and (2.11), we have
\begin{align}
\nabla^{1,\alpha}_{\dot{\gamma}}\dot{\gamma}&=
\left[\frac{\ddot{\gamma}_1\gamma_1-(\dot{\gamma}_1)^2}{\gamma_1^2}+L\omega(\dot{\gamma}(t))\left(\frac{(1-\alpha)\dot{\gamma_2}(t)}{\gamma_1}+\frac{\alpha\dot{\gamma_3}(t)}{2}\right)\right]X_1\nonumber\\
&+\left[\ddot{\gamma}_3-\frac{(2-\alpha)\dot{\gamma}_1L}{2\gamma_1}\omega(\dot{\gamma}(t)\right]X_2+\left[\frac{d}{dt}\omega(\dot{\gamma}(t))-\frac{(1-\alpha)\dot{\gamma}_1}{\gamma_1}\omega(\dot{\gamma}(t))\right]X_3.\nonumber\\
\end{align}
By (2.7), (2.10) and (2.12), we get Lemma 2.4.
\end{proof}
Similarly,
\begin{defn}(\cite{YS})
Let $\gamma:[a,b]\rightarrow (\mathbb{M},g_L)$ be a Euclidean $C^2$-smooth regular curve in the Riemannian manifold $(\mathbb{M},g_L)$.
We define the intrinsic curvature $k_{\gamma}^{\infty,\nabla^{1,\alpha}}$ of $\gamma$ at $\gamma(t)$ to be
$$k_{\gamma}^{\infty,\nabla^{1,\alpha}}:={\rm lim}_{L\rightarrow +\infty}k_{\gamma}^{L,\nabla^{1,\alpha}},$$
if the limit exists.
\end{defn}
\indent We introduce the following notation: for continuous functions $F_1,F_2:(0,+\infty)\rightarrow \mathbb{R}$,
\begin{equation}
F_1(L)\sim F_2(L),~~as ~~L\rightarrow +\infty\Leftrightarrow {\rm lim}_{L\rightarrow +\infty}\frac{F_1(L)}{F_2(L)}=1.
\end{equation}
Then, we have
\begin{lem}
Let $\gamma:[a,b]\rightarrow (\mathbb{M},g_L)$ be a Euclidean $C^2$-smooth regular curve in the Riemannian manifold $(\mathbb{M},g_L)$. Then\\
(1)when $\omega(\dot{\gamma}(t))\neq 0$,\\
\begin{equation}
k_{\gamma}^{\infty,\nabla^{1,\alpha}}=\frac{\sqrt{\left(\frac{(1-\alpha)\dot{\gamma}_2}{\gamma_1(t)}+\frac{\alpha}{2}\dot{\gamma}_3\right)^2+\left(\frac{(2-\alpha)\dot{\gamma}_1}{2\gamma_1}\right)^2}}{|\omega(\dot{\gamma}(t))|},
\end{equation}
(2)when $\omega(\dot{\gamma}(t))= 0 ~~~and~~~\frac{d}{dt}(\omega(\dot{\gamma}(t)))=0$,\\
\begin{align}
k^{\infty,\nabla^{1,\alpha}}_{\gamma}&=\Bigg\{\left\{\left[\frac{\ddot{\gamma}_1\gamma_1-\dot{\gamma}_1^2}{\gamma_1^2}\right]^2+\dot{\gamma}_3^2\right\}
\cdot\left[\left(\frac{\dot{\gamma}_1}{\gamma_1}\right)^2+\dot{\gamma}_3^2\right]^{-2}\nonumber\\
&-\left[\frac{\gamma_1\dot{\gamma}_1\ddot{\gamma}_1-\dot{\gamma}_1^3}{\gamma_1^3}+\dot{\gamma}_3\ddot{\gamma}_3\right]^2
\cdot\left[\left(\frac{\dot{\gamma}_1}{\gamma_1}\right)^2+\dot{\gamma}_3^2\right]^{-3}\Bigg\}^{\frac{1}{2}},\nonumber\\
\end{align}
(3)when $\omega(\dot{\gamma}(t))= 0 ~~~and~~~\frac{d}{dt}(\omega(\dot{\gamma}(t)))\neq0$,\\
\begin{equation}
{\rm lim}_{L\rightarrow +\infty}\frac{k_{\gamma}^{L,\nabla^{1,\alpha}}}{\sqrt{L}}=\frac{|\frac{d}{dt}(\omega(\dot{\gamma}(t)))|}{\left(\frac{\dot{\gamma}_1}{\gamma_1}\right)^2+\dot{\gamma}_3^2}.
\end{equation}
\end{lem}
\begin{proof}
By (2.13), when $\omega(\dot{\gamma}(t))\neq 0$, we have
$$||\nabla^{1,\alpha}_{\dot{\gamma}}{\dot{\gamma}}||_L^2\sim L^2\omega(\dot{\gamma}(t))^2\left[\left(\frac{(1-\alpha)\dot{\gamma}_2}{\gamma_1}+\frac{\alpha\dot{\gamma}_3}{2}\right)^2+\left(\frac{(2-\alpha)\dot{\gamma}_1}{2\gamma_1}\right)^2\right],~~as~~L\rightarrow +\infty,$$
$$||\dot{\gamma}||^2_L=\left(\frac{\dot{\gamma}_1}{\gamma_1}\right)^2+\dot{\gamma}_3^2+L\omega(\dot{\gamma}(t))^2\sim L\omega(\dot{\gamma}(t))^2,~~as~~L\rightarrow +\infty,$$
$$\langle \nabla^{1,\alpha}_{\dot{\gamma}}{\dot{\gamma}},\dot{\gamma}\rangle^2_L\sim O(L^2)~~as~~L\rightarrow +\infty.$$
Therefore
$$\frac{||\nabla^{1,\alpha}_{\dot{\gamma}}{\dot{\gamma}}||_L^2}{||\dot{\gamma}||^4_L}\rightarrow \frac{\left(\frac{(1-\alpha)\dot{\gamma}_2}{\gamma_1}+\frac{\alpha\dot{\gamma}_3}{2}\right)^2+\left(\frac{(2-\alpha)\dot{\gamma}_1}{2\gamma_1}\right)^2}{\omega(\dot{\gamma}(t))^2},~~as~~L\rightarrow +\infty,$$
$$\frac{\langle \nabla^{1,\alpha}_{\dot{\gamma}}{\dot{\gamma}},\dot{\gamma}\rangle^2_L}{||\dot{\gamma}||^6_L}\rightarrow 0,~~as~~L\rightarrow +\infty.$$
So by (2.7), we have (2.14).\\
Obviously, when $\omega(\dot{\gamma}(t))= 0$
~~~and~~~$\frac{d}{dt}(\omega(\dot{\gamma}(t)))=0,$ we have (2.15).\\
 When
        $\omega(\dot{\gamma}(t))= 0$
~~~and~~~$\frac{d}{dt}(\omega(\dot{\gamma}(t)))\neq 0,$\\
 we have
$$||\nabla^{1,\alpha}_{\dot{\gamma}}{\dot{\gamma}}||_L^2\sim L[\frac{d}{dt}(\omega(\dot{\gamma}(t)))]^2,~~as~~L\rightarrow +\infty,$$
$$||\dot{\gamma}||^2_L=\left(\frac{\dot{\gamma}_1}{\gamma_1}\right)^2+\dot{\gamma}_3^2,$$
$$\langle \nabla^{1,\alpha}_{\dot{\gamma}}{\dot{\gamma}},\dot{\gamma}\rangle^2_L=O(1)~~as~~L\rightarrow +\infty.$$
By (2.7), we get (2.16).
\end{proof}
\begin{defn}(\cite{YS})
 If $\Sigma\subset(\mathbb{M},g_L)$ is a Euclidean $C^2$-smooth compact and oriented surface, then this surface $\Sigma\subset(\mathbb{M},g_L)$ is regular.
\end{defn}
First we assume that there exists
a Euclidean $C^2$-smooth function $u:\mathbb{M}\rightarrow \mathbb{R}$ which satisfies
$\Sigma=\{(x_1,x_2,x_3)\in \mathbb{M}:u(x_1,x_2,x_3)=0\}$
and $u_{x_1}\partial_{x_1}+u_{x_2}\partial_{x_2}+u_{x_3}\partial_{x_3}\neq 0.$\\
\begin{defn}(\cite{YS})
 A point $x\in\Sigma$ is called {\it characteristic} if $\nabla_Hu(x)=0$, where $\nabla_Hu=X_1(u)X_1+X_2(u)X_2.$
\end{defn}
Then we have the characteristic set $C(\Sigma):=\{x\in\Sigma|\nabla_Hu(x)=0\}.$ Nextly, our computations will
be local and away from characteristic points of $\Sigma$. In order to facilitate the next calculation, let us define first
$$p:=X_1u,~~~~q:=X_2u ,~~{\rm and}~~r:=\widetilde{X}_3u.$$
We then define
\begin{align}
&l:=\sqrt{p^2+q^2},~~~~l_L:=\sqrt{p^2+q^2+r^2},~~~~\overline{p}:=\frac{p}{l},\\
&\overline{q}:=\frac{q}{l},~~~~
\overline{p_L}:=\frac{p}{l_L},~~~~\overline{q_L}:=\frac{q}{l_L},~~~~\overline{r_L}:=\frac{r}{l_L}.\notag
\end{align}
In particular, $\overline{p}^2+\overline{q}^2=1$. These functions are well defined at every non-characteristic point. Let
\begin{align}
v_L=\overline{p_L}X_1+\overline{q_L}X_2+\overline{r_L}\widetilde{X_3},~~~~e_1=\overline{q}X_1-\overline{p}X_2,~~~~
e_2=\overline{r_L}~~\overline{p}X_1+\overline{r_L}~~ \overline{q}X_2-\frac{l}{l_L}\widetilde{X_3},
\end{align}
then $v_L$ is the Riemannian unit normal vector to $\Sigma$ and $e_1,e_2$ are the orthonormal basis of $\Sigma$. On $T\Sigma$ we define a linear transformation $J_L:T\Sigma\rightarrow T\Sigma$ such that
\begin{equation}
J_L(e_1):=e_2;~~~~J_L(e_2):=-e_1.
\end{equation}
For every $U,V\in T\Sigma$, we define $\nabla^{\Sigma,{1,\alpha}}_UV=\pi \nabla^{1,\alpha}_UV$ where $\pi:T\mathbb{M}\rightarrow T\Sigma$ is the projection. Then $\nabla^{\Sigma,{1,\alpha}}$ is the Levi-Civita connection on $\Sigma$
with respect to the metric $g_L$. By (2.12),(2.18) and
\begin{equation}
\nabla^{\Sigma,{1,\alpha}}_{\dot{\gamma}}\dot{\gamma}=\langle \nabla^{1,\alpha}_{\dot{\gamma}}\dot{\gamma},e_1\rangle_Le_1+\langle \nabla^{1,\alpha}_{\dot{\gamma}}\dot{\gamma},e_2\rangle_Le_2,
\end{equation}
we have
\begin{align}
\nabla^{\Sigma,{1,\alpha}}_{\dot{\gamma}}\dot{\gamma}&=
\Bigg\{\overline{q}\left[\frac{\ddot{\gamma}_1\gamma_1-(\dot{\gamma}_1)^2}{\gamma_1^2}+L\omega(\dot{\gamma}(t))\left(\frac{(1-\alpha)\dot{\gamma_2}(t)}{\gamma_1}+\frac{\alpha\dot{\gamma_3}(t)}{2}\right)\right]-\overline{p}\left[\ddot{\gamma}_3-\frac{(2-\alpha)\dot{\gamma}_1L}{2\gamma_1}\omega(\dot{\gamma}(t)\right]\Bigg\}e_1\nonumber\\
&+\Bigg\{\overline{r_L}~~\overline{p}\left[\frac{\ddot{\gamma}_1\gamma_1-(\dot{\gamma}_1)^2}{\gamma_1^2}+L\omega(\dot{\gamma}(t))\left(\frac{(1-\alpha)\dot{\gamma_2}(t)}{\gamma_1}+\frac{\alpha\dot{\gamma_3}(t)}{2}\right)\right]+\overline{r_L}~~ \overline{q}\left[\ddot{\gamma}_3-\frac{(2-\alpha)\dot{\gamma}_1L}{2\gamma_1}\omega(\dot{\gamma}(t)\right]\nonumber\\
&-\frac{l}{l_L}L^{\frac{1}{2}}\left[\frac{d}{dt}\omega(\dot{\gamma}(t))-\frac{(1-\alpha)\dot{\gamma}_1}{\gamma_1}\omega(\dot{\gamma}(t))\right]\Bigg\}e_2.\nonumber\\
\end{align}
Moreover if $\omega(\dot{\gamma}(t))=0$, then
\begin{align}
\nabla^{\Sigma,{1,\alpha}}_{\dot{\gamma}}\dot{\gamma}&=
\Bigg\{\overline{q}\left[\frac{\ddot{\gamma}_1\gamma_1-(\dot{\gamma}_1)^2}{\gamma_1^2}\right]-\overline{p}\ddot{\gamma}_3\Bigg\}e_1+\Bigg\{\overline{r_L}~~\overline{p}\frac{\ddot{\gamma}_1\gamma_1-(\dot{\gamma}_1)^2}{\gamma_1^2}+\overline{r_L}~~ \overline{q}
\ddot{\gamma}_3-\frac{l}{l_L}L^{\frac{1}{2}}\frac{d}{dt}\omega(\dot{\gamma}(t))\Bigg\}e_2.\nonumber\\
\end{align}
\begin{defn}(\cite{YS})
Let $\Sigma\subset(\mathbb{M},g_L)$ be a regular surface.
Let $\gamma:[a,b]\rightarrow \Sigma$ be a Euclidean $C^2$-smooth regular curve. The geodesic curvature $k^{L,\nabla^{1,\alpha}}_{\gamma,\Sigma}$ of $\gamma$ at $\gamma(t)$ is defined as
\begin{equation}
k^{L,\nabla^{1,\alpha}}_{\gamma,\Sigma}:=\sqrt{\frac{||\nabla^{\Sigma,{1,\alpha}}_{\dot{\gamma}}{\dot{\gamma}}||_{\Sigma,L}^2}{||\dot{\gamma}||^4_{\Sigma,L}}-\frac{\langle \nabla^{\Sigma,{1,\alpha}}_{\dot{\gamma}}{\dot{\gamma}},\dot{\gamma}\rangle^2_{\Sigma,{1,\alpha}}}{||\dot{\gamma}||^6_{\Sigma,L}}}.
\end{equation}
\end{defn}
\begin{defn}(\cite{YS})
Let $\Sigma\subset(\mathbb{M},g_L)$ be a regular surface. Let $\gamma:[a,b]\rightarrow \Sigma$ be a Euclidean $C^2$-smooth regular curve.
We define the intrinsic geodesic curvature $k_{\gamma,\Sigma}^{\infty,{1,\alpha}}$ of $\gamma$ at $\gamma(t)$ to be
$$k_{\gamma,\Sigma}^{\infty,\nabla^{1,\alpha}}:={\rm lim}_{L\rightarrow +\infty}k_{\gamma,\Sigma}^{L,\nabla^{1,\alpha}},$$
if the limit exists.
\end{defn}
\vskip 0.5 true cm
\begin{lem}
Let $\Sigma\subset(\mathbb{M},g_L)$ be a regular surface.
Let $\gamma:[a,b]\rightarrow \Sigma$ be a Euclidean $C^2$-smooth regular curve. Then\\
(1)when $\omega(\dot{\gamma}(t))\neq 0,$
\begin{equation}
k_{\gamma,\Sigma}^{\infty,\nabla^{1,\alpha}}=\frac{|\overline{p}\frac{(2-\alpha)\dot{\gamma}_1}{2\gamma_1}+\overline{q}\left(\frac{(1-\alpha)\dot{\gamma}_2}{\gamma_1}+\frac{\alpha\dot{\gamma}_3}{2}\right)|}{|\omega(\dot{\gamma}(t))|},
\end{equation}
(2)when $\omega(\dot{\gamma}(t))= 0, ~~~and~~~\frac{d}{dt}(\omega(\dot{\gamma}(t)))=0,$
\begin{equation}
k_{\gamma,\Sigma}^{\infty,\nabla^{1,\alpha}}=0,\\
\end{equation}
(3)when $\omega(\dot{\gamma}(t))= 0, ~~~and~~~\frac{d}{dt}(\omega(\dot{\gamma}(t)))\neq0,$
\begin{equation}
{\rm lim}_{L\rightarrow +\infty}\frac{k_{\gamma,\Sigma}^{L,\nabla^{1,\alpha}}}{\sqrt{L}}=\frac{|\frac{d}{dt}(\omega(\dot{\gamma}(t)))|}{\left(\overline{q}\frac{\dot{\gamma}_1}{\gamma_1}-\overline{p}\dot{\gamma}_3\right)^2}.
\end{equation}
\end{lem}

\begin{proof}  We know $\dot{\gamma}(t)=\frac{\dot{\gamma_1}(t)}{\gamma_1(t)}X_1+\gamma_3(t)X_2+\omega(\dot{\gamma}(t))X_3,$
let $$\dot{\gamma}(t)=\lambda_1e_1+\lambda_2e_2.$$
Then
\begin{eqnarray}
       \begin{cases}
\frac{\dot{\gamma_1}(t)}{\gamma_1(t)}=\lambda_1\overline{q}+\lambda_2\overline{r_L}\overline{p}\\[2pt]
\dot{\gamma_3}(t)=-\lambda_1\overline{p}+\lambda_2\overline{r_L}\overline{q}\\[2pt]
\omega(\dot{\gamma}(t))=-\lambda_2\frac{l}{l_L}L^{-\frac{1}{2}},\\[2pt]
 \end{cases}
\end{eqnarray}
we have
\begin{eqnarray}
       \begin{cases}
\lambda_1=\overline{q}\frac{\dot{\gamma_1}(t)}{\gamma_1(t)}-\overline{p}\dot{\gamma_3}(t)\\[2pt]
\lambda_2=-\lambda_2\frac{l_L}{l}L^{\frac{1}{2}}\omega(\dot{\gamma}(t)).\\[2pt]
 \end{cases}
\end{eqnarray}
Thus $\dot{\gamma}\in T\Sigma$, we have
\begin{equation}
\dot{\gamma}=(\overline{q}\frac{\dot{\gamma}_1}{\gamma_1}-\overline{p}\dot{\gamma}_3)e_1-\frac{l_L}{l}L^{\frac{1}{2}}\omega(\dot{\gamma}(t))e_2.
\end{equation}
By (2.21), we have
\begin{align}
||\nabla^{\Sigma,1,\alpha}_{\dot{\gamma}}\dot{\gamma}||^2_{L,\Sigma}&=
\Bigg\{\overline{q}\left[\frac{\ddot{\gamma}_1\gamma_1-(\dot{\gamma}_1)^2}{\gamma_1^2}+L\omega(\dot{\gamma}(t))\left(\frac{(1-\alpha)\dot{\gamma_2}(t)}{\gamma_1}+\frac{\alpha\dot{\gamma_3}(t)}{2}\right)\right]-\overline{p}\left[\ddot{\gamma}_3-\frac{(2-\alpha)\dot{\gamma}_1L}{2\gamma_1}\omega(\dot{\gamma}(t)\right]\Bigg\}^2\nonumber\\
&+\Bigg\{\overline{r_L}~~\overline{p}\left[\frac{\ddot{\gamma}_1\gamma_1-(\dot{\gamma}_1)^2}{\gamma_1^2}+L\omega(\dot{\gamma}(t))\left(\frac{(1-\alpha)\dot{\gamma_2}(t)}{\gamma_1}+\frac{\alpha\dot{\gamma_3}(t)}{2}\right)\right]\nonumber\\
&+\overline{r_L}~~ \overline{q}\left[\ddot{\gamma}_3-\frac{(2-\alpha)\dot{\gamma}_1L}{2\gamma_1}\omega(\dot{\gamma}(t))\right]-\frac{l}{l_L}L^{\frac{1}{2}}\left[\frac{d}{dt}\omega(\dot{\gamma}(t))-\frac{(1-\alpha)\dot{\gamma}_1}{\gamma_1}\omega(\dot{\gamma}(t))\right]\Bigg\}^2\nonumber\\
&\sim L^2\left[\overline{p}\frac{(2-\alpha)\dot{\gamma}_1}{2\gamma_1}+\overline{q}\left(\frac{(1-\alpha)\dot{\gamma}_2}{\gamma_1}+\frac{\alpha\dot{\gamma}_3}{2}\right)\right]^2\omega(\dot{\gamma}(t))^2,~~~~ {\rm as} ~~~~L\rightarrow +\infty.\nonumber\\
\end{align}
Similarly, when $\omega(\dot{\gamma}(t))\neq 0$,
\begin{equation}
||\dot{\gamma}||_{\Sigma,L}=\sqrt{(\overline{q}\frac{\dot{\gamma}_1}{\gamma_1}-\overline{p}\dot{\gamma}_3)^2+(\frac{l_L}{l})^2L\omega(\dot{\gamma}(t))^2}\sim L^{\frac{1}{2}}|\omega(\dot{\gamma}(t))|
,~~~~ {\rm as} ~~~~L\rightarrow +\infty.
\end{equation}
By (2.21) and (2.29), we have
\begin{align}
&\langle \nabla^{\Sigma,1,\alpha}_{\dot{\gamma}}{\dot{\gamma}},\dot{\gamma}\rangle_{\Sigma,L}\nonumber\\
&=
\left(\overline{q}\frac{\dot{\gamma}_1}{f}
-\overline{p}\dot{\gamma}_3\right)\cdot
\Bigg\{\overline{q}\left[\frac{\ddot{\gamma}_1\gamma_1-(\dot{\gamma}_1)^2}{\gamma_1^2}+L\omega(\dot{\gamma}(t))\left(\frac{(1-\alpha)\dot{\gamma_2}(t)}{\gamma_1}+\frac{\alpha\dot{\gamma_3}(t)}{2}\right)\right]-\overline{p}\left[\ddot{\gamma}_3-\frac{(2-\alpha)\dot{\gamma}_1L}{2\gamma_1}\omega(\dot{\gamma}(t)\right]\Bigg\}\nonumber\\
&-\frac{l_L}{l}L^{\frac{1}{2}}\omega(\dot{\gamma}(t))
\cdot\Bigg\{\overline{r_L}~~\overline{p}\left[\frac{\ddot{\gamma}_1\gamma_1-(\dot{\gamma}_1)^2}{\gamma_1^2}+L\omega(\dot{\gamma}(t))\left(\frac{(1-\alpha)\dot{\gamma_2}(t)}{\gamma_1}+\frac{\alpha\dot{\gamma_3}(t)}{2}\right)\right]\nonumber\\
&+\overline{r_L}~~ \overline{q}\left[\ddot{\gamma}_3-\frac{(2-\alpha)\dot{\gamma}_1L}{2\gamma_1}\omega(\dot{\gamma}(t))\right]-\frac{l}{l_L}L^{\frac{1}{2}}\left[\frac{d}{dt}\omega(\dot{\gamma}(t))-\frac{(1-\alpha)\dot{\gamma}_1}{\gamma_1}\omega(\dot{\gamma}(t))\right]\Bigg\}\nonumber\\
&\sim C_0L,
\end{align}
where $C_0$ does not depend on $L$. By (2.23),(2.30)-(2.32), we get (2.24).\\
When $\omega(\dot{\gamma}(t))= 0$~~~and~~~$\frac{d}{dt}(\omega(\dot{\gamma}(t)))=0,$\\
 we have\\
\begin{align}
||\nabla^{\Sigma,1,\alpha}_{\dot{\gamma}}\dot{\gamma}||^2_{L,\Sigma}&=
\left(\overline{q}\frac{\ddot{\gamma}_1\gamma_1-(\dot{\gamma}_1)^2}{\gamma_1^2}-\overline{p}\ddot{\gamma}_3\right)^2+\left(\overline{r_L}~~\overline{p}\frac{\ddot{\gamma}_1\gamma_1-(\dot{\gamma}_1)^2}{\gamma_1^2}
+\overline{r_L}~~ \overline{q}
\ddot{\gamma}_3\right)^2\nonumber\\
&\sim \left(\overline{q}\frac{\ddot{\gamma}_1\gamma_1-(\dot{\gamma}_1)^2}{\gamma_1^2}-\overline{p}\ddot{\gamma}_3\right)^2\nonumber\\
\end{align}
and
\begin{equation}
||\dot{\gamma}||_{\Sigma,L}=|\overline{q}\frac{\dot{\gamma}_1}{\gamma_1}-\overline{p}\dot{\gamma}_3|,
\end{equation}
\begin{align}
\langle \nabla^{\Sigma,1,\alpha}_{\dot{\gamma}}{\dot{\gamma}},\dot{\gamma}\rangle_{\Sigma,L}&=
(\overline{q}\frac{\dot{\gamma}_1}{\gamma_1}
-\overline{p}\dot{\gamma}_3)
\cdot\left(\overline{q}\frac{\ddot{\gamma}_1\gamma_1-(\dot{\gamma}_1)^2}{\gamma_1^2}-\overline{p}\ddot{\gamma}_3\right)\nonumber\\
\end{align}
By (2.33)-(2.35) and (2.23), we get $k^{\infty,1,\alpha}_{\gamma,\Sigma}=0$.\\
When $\omega(\dot{\gamma}(t))= 0$~~~and~~~$\frac{d}{dt}(\omega(\dot{\gamma}(t)))\neq 0,$
 we have\\
$$||\nabla^{\Sigma,L}_{\dot{\gamma}}\dot{\gamma}||^2_{L,\Sigma}\sim L
[\frac{d}{dt}(\omega(\dot{\gamma}(t)))]^2,$$
$$\langle \nabla^{\Sigma,L}_{\dot{\gamma}}{\dot{\gamma}},\dot{\gamma}\rangle_{\Sigma,L}=O(1),$$
so we get (2.26).
\end{proof}
\begin{defn}(\cite{YS})
Let $\Sigma\subset(\mathbb{M},g_L)$ be a regular surface.
Let $\gamma:[a,b]\rightarrow \Sigma$ be a Euclidean $C^2$-smooth regular curve. The signed geodesic curvature $k^{L,1,\alpha,s}_{\gamma,\Sigma}$ of $\gamma$ at $\gamma(t)$ is defined as
\begin{equation}
k^{L,\nabla^{1,\alpha},s}_{\gamma,\Sigma}:=\frac{\langle \nabla^{\Sigma,1,\alpha}_{\dot{\gamma}}{\dot{\gamma}},J_L(\dot{\gamma})\rangle_{\Sigma,L}}{||\dot{\gamma}||^3_{\Sigma,L}},
\end{equation}
where $J_L$ is defined by (2.19).
\end{defn}
\begin{defn}(\cite{YS})
Let $\Sigma\subset(\mathbb{M},g_L)$ be a regular surface and let $\gamma:[a,b]\rightarrow \Sigma$ be a Euclidean $C^2$-smooth regular curve.
We define the intrinsic geodesic curvature $k_{\gamma,\Sigma}^{\infty,\nabla^{1,\alpha},s}$ of $\gamma$ at the non-characteristic point $\gamma(t)$ to be
$$k_{\gamma,\Sigma}^{\infty,\nabla^{1,\alpha},s}:={\rm lim}_{L\rightarrow +\infty}k_{\gamma,\Sigma}^{L,\nabla^{1,\alpha},s},$$
if the limit exists.
\end{defn}
\begin{lem}
Let $\Sigma\subset(\mathbb{M},g_L)$ be a regular surface.
Let $\gamma:[a,b]\rightarrow \Sigma$ be a Euclidean $C^2$-smooth regular curve. Then\\
(1)when $\omega(\dot{\gamma}(t))\neq 0$,
\begin{equation}
k_{\gamma,\Sigma}^{\infty,\nabla^{1,\alpha},s}=\frac{|\overline{p}\frac{(2-\alpha)\dot{\gamma}_1}{2\gamma_1}+\overline{q}\left(\frac{(1-\alpha)\dot{\gamma}_2}{\gamma_1}+\frac{\alpha\dot{\gamma}_3}{2}\right)|}{|\omega(\dot{\gamma}(t))|},
\end{equation}
(2)when $\omega(\dot{\gamma}(t))= 0, ~~ ~and~ ~~\frac{d}{dt}(\omega(\dot{\gamma}(t)))=0$,
\begin{equation}
k^{\infty,\nabla^{1,\alpha},s}_{\gamma,\Sigma}=0,\\
\end{equation}
(3)when $\omega(\dot{\gamma}(t))= 0, ~~ ~and~ ~~\frac{d}{dt}(\omega(\dot{\gamma}(t)))\neq0$,
\begin{equation}
{\rm lim}_{L\rightarrow +\infty}\frac{k_{\gamma,\Sigma}^{L,\nabla^{1,\alpha},s}}{\sqrt{L}}=\frac{(-\overline{q}\frac{\dot{\gamma}_1}{\gamma_1}+\overline{p}\dot{\gamma}_3)\frac{d}{dt}(\omega(\dot{\gamma}(t)))}{|\overline{q}\frac{\dot{\gamma}_1}{\gamma_1}-\overline{p}\dot{\gamma}_3|^3}.
\end{equation}
\end{lem}
\begin{proof} By (2.19) and (2.29), we have
\begin{equation}
J_L(\dot{\gamma})=\frac{l_L}{l}L^{\frac{1}{2}}\omega(\dot{\gamma}(t))e_1+(\overline{q}\frac{\dot{\gamma}_1}{\gamma_1}-\overline{p}\dot{\gamma}_3)e_2.
\end{equation}
By (2.21) and (2.40), we have
\begin{align}
&\langle \nabla^{\Sigma,1,\alpha}_{\dot{\gamma}}\dot{\gamma},J_L(\dot{\gamma})\rangle_{L,\Sigma}\nonumber\\
&=
\frac{l_L}{l}L^{\frac{1}{2}}\omega(\dot{\gamma}(t))
\Bigg\{\overline{q}\left[\frac{\ddot{\gamma}_1\gamma_1-(\dot{\gamma}_1)^2}{\gamma_1^2}+L\omega(\dot{\gamma}(t))\left(\frac{(1-\alpha)\dot{\gamma_2}(t)}{\gamma_1}+\frac{\alpha\dot{\gamma_3}(t)}{2}\right)\right]-\overline{p}\left[\ddot{\gamma}_3-\frac{(2-\alpha)\dot{\gamma}_1L}{2\gamma_1}\omega(\dot{\gamma}(t)\right]\Bigg\}\nonumber\\
&+(\overline{q}\frac{\dot{\gamma}_1}{f}-\overline{p}\dot{\gamma}_3)\cdot
\Bigg\{\overline{r_L}~~\overline{p}\left[\frac{\ddot{\gamma}_1\gamma_1-(\dot{\gamma}_1)^2}{\gamma_1^2}+L\omega(\dot{\gamma}(t))\left(\frac{(1-\alpha)\dot{\gamma_2}(t)}{\gamma_1}+\frac{\alpha\dot{\gamma_3}(t)}{2}\right)\right]\nonumber\\
&+\overline{r_L}~~ \overline{q}\left[\ddot{\gamma}_3-\frac{(2-\alpha)\dot{\gamma}_1L}{2\gamma_1}\omega(\dot{\gamma}(t))\right]-\frac{l}{l_L}L^{\frac{1}{2}}\left[\frac{d}{dt}\omega(\dot{\gamma}(t))-\frac{(1-\alpha)\dot{\gamma}_1}{\gamma_1}\omega(\dot{\gamma}(t))\right]\Bigg\},\nonumber\\
&\sim L^{\frac{3}{2}}\omega(\dot{\gamma}(t))^2\left[\overline{p}\frac{(2-\alpha)\dot{\gamma}_1}{2\gamma_1}+\overline{q}\left(\frac{(1-\alpha)\dot{\gamma}_2}{\gamma_1}+\frac{\alpha\dot{\gamma}_3}{2}\right)\right]~~{\rm as}~~L\rightarrow +\infty.\nonumber\\
\end{align}
So we get (2.37).\\
 When
$\omega(\dot{\gamma}(t))= 0$~~~and~~~$\frac{d}{dt}(\omega(\dot{\gamma}(t)))=0,$
we get
\begin{align}
&\langle \nabla^{\Sigma,1,\alpha}_{\dot{\gamma}}\dot{\gamma},J_L(\dot{\gamma})\rangle_{L,\Sigma}\nonumber\\
&=(\overline{q}\frac{\dot{\gamma}_1}{f}-\overline{p}\dot{\gamma}_3)\cdot
\Bigg\{\overline{r_L}~~\overline{p}\left[\frac{\ddot{\gamma}_1\gamma_1-(\dot{\gamma}_1)^2}{\gamma_1^2}+L\omega(\dot{\gamma}(t))\left(\frac{(1-\alpha)\dot{\gamma_2}(t)}{\gamma_1}+\frac{\alpha\dot{\gamma_3}(t)}{2}\right)\right]\nonumber\\
&+\overline{r_L}~~ \overline{q}\left[\ddot{\gamma}_3-\frac{(2-\alpha)\dot{\gamma}_1L}{2\gamma_1}\omega(\dot{\gamma}(t))\right]-\frac{l}{l_L}L^{\frac{1}{2}}\left[\frac{d}{dt}\omega(\dot{\gamma}(t))-\frac{(1-\alpha)\dot{\gamma}_1}{\gamma_1}\omega(\dot{\gamma}(t))\right]\Bigg\}\nonumber\\
&\sim C_1L^{-\frac{1}{2}}~~{\rm as}~~L\rightarrow +\infty.
\end{align}
So $k^{\infty,\nabla^{1,\alpha},s}_{\gamma,\Sigma}=0.$\\
 When
$\omega(\dot{\gamma}(t))=0$~~~and~~~$\frac{d}{dt}(\omega(\dot{\gamma}(t)))\neq0,$
 we have\\
\begin{align}
\langle \nabla^{\Sigma,1,\alpha}_{\dot{\gamma}}\dot{\gamma},J_L(\dot{\gamma})\rangle_{L,\Sigma}\sim
L^{\frac{1}{2}}(-\overline{q}\frac{\dot{\gamma}_1}{\gamma_1}+\overline{p}\dot{\gamma}_3)\frac{d}{dt}(\omega(\dot{\gamma}(t)))~~{\rm as}~~L\rightarrow +\infty.
\end{align}
So we get (2.39).
\end{proof}

In the following, we compute the sub-Riemannian limit of the Riemannian Gaussian curvature of surfaces associated to the first kind of deformed Schouten-Van Kampen connection in the affine group. We define the {\it second fundamental form} $II^{\nabla^{1,\alpha},L}$ of the
embedding of $\Sigma$ into $(\mathbb{M},g_L)$:
\begin{equation}
II^{\nabla^{1,\alpha},L}=\left(
  \begin{array}{cc}
   \langle \nabla^{1,\alpha}_{e_1}v_L,e_1\rangle_{L},
    & \langle \nabla^{1,\alpha}_{e_1}v_L,e_2\rangle_{L} \\
   \langle \nabla^{1,\alpha}_{e_2}v_L,e_1\rangle_{L},
    & \langle \nabla^{1,\alpha}_{e_2}v_L,e_2\rangle_{L} \\
  \end{array}
\right).
\end{equation}
Similarly to Theorem 4.3 in \cite{CDPT}, we have
\vskip 0.5 true cm
\begin{thm} The second fundamental form $II^{\nabla^{1,\alpha},L}$ of the
embedding of $\Sigma$ into $(\mathbb{M},g_L)$ is given by
\begin{equation}
II^{\nabla^{1,\alpha},L}=\left(
  \begin{array}{cc}
  h_{11},
    & h_{12} \\
  h_{21} ,
    & h_{22} \\
  \end{array}
\right),
\end{equation}
where
$$h_{11}=\frac{l}{l_L}[X_1(\overline{p})+X_2(\overline{q})],$$
    $$h_{12}=-\frac{l_L}{l}\langle e_1,\nabla_H(\overline{r_L})\rangle_L-\frac{(1-\alpha)\sqrt{L}}{2},$$
    $$h_{21}=-\frac{l_L}{l}\langle e_1,\nabla_H(\overline{r_L})\rangle_L-\frac{(1+\alpha\overline{r_L}^2)\sqrt{L}}{2}+\alpha\overline{q_L}\overline{r_L},$$
    $$h_{22}=-\frac{l^2}{l_L^2}\langle e_2,\nabla_H(\frac{r}{l})\rangle_L+\widetilde{X_3}(\overline{r_L})-(1-\alpha)\overline{p_L}.$$
\end{thm}
\begin{proof}
By $e_i\langle V_L,e_j\rangle_L-\langle\nabla^L_{e_i}V_L,e_j\rangle_L-\langle\nabla^L_{e_i}e_j,V_L\rangle_L=0$ and $e_i\langle V_L,e_j\rangle_L=0$, we have $\langle\nabla^L_{e_i}V_L,e_j\rangle_L=-\langle\nabla^L_{e_i}e_j,V_L\rangle_L,$ $i,j=1,2$.\\
By lemma 2.1 and (2.18),
\begin{align}
\nabla^{1,\alpha}_{e_1}e_1&=\nabla^{1,\alpha}_{(\overline{q}X_1-\overline{p}X_2)}(\overline{q}X_1-\overline{p}X_2)\nonumber\\
&=[\overline{q}X_1(\overline{q})-\overline{p}X_2(\overline{p})]X_1-[\overline{q}X_1(\overline{p})-\overline{p}X_2(\overline{p})]X_2.\nonumber\\
\end{align}
Then
\begin{align}
h_{11}&=-\langle\nabla^{1,\alpha}_{e_1}e_1,V_L\rangle_L\nonumber\\
&=-\overline{p_L}[\overline{q}X_1(\overline{q})-\overline{p}X_2(\overline{p})]+\overline{q_L}[\overline{q}X_1(\overline{p})-\overline{p}X_2(\overline{p})]\nonumber\\
&=\frac{l}{l_L}[X_1(\overline{p})+X_2(\overline{q})].\nonumber\\
\end{align}
Similarly,
\begin{align}
\nabla^{1,\alpha}_{e_1}e_2&=\nabla^L_{(\overline{q}X_1-\overline{p}X_2)}(\overline{r_L}\overline{p}X_1+\overline{r_L}\overline{q}X_2-\frac{l}{l_L}\widetilde{X_3})\nonumber\\
&=[\overline{q}X_1(\overline{r_L}\overline{p})-\overline{p}X_2(\overline{r_L}\overline{p})+\frac{(1-\alpha)\overline{p_L}\sqrt{L}}{2}]X_1\nonumber\\
&+[\overline{q}X_1(\overline{q})-\overline{p}X_2(\overline{q})+\frac{(1-\alpha)\overline{q_L}\sqrt{L}}{2}]X_2\nonumber\\
&+[\frac{(1-\alpha)\overline{r_L}\sqrt{L}}{2}+\overline{p}X_2(\frac{l}{l_L})-\overline{q}X_2(\frac{l}{l_L})]\widetilde{X_3}.\nonumber\\
\end{align}
Then
\begin{align}
h_{12}&=-\langle\nabla^{1,\alpha}_{e_1}e_2,V_L\rangle_L\nonumber\\
&=-\frac{l}{l_L}[\overline{q}X_1(\overline{r_L})-\overline{p}X_2(\overline{r_L})]+\overline{r_L}[\overline{q}X_1(\frac{l}{l_L})-\overline{p}X_2(\frac{l}{l_L})]-\frac{(1-\alpha)\sqrt{L}}{2}\nonumber\\
&=-\frac{l_L}{l}\langle e_1,\nabla_H(\overline{r_L})\rangle_L-\frac{(1-\alpha)\sqrt{L}}{2}.\nonumber\\
\end{align}
Since
\begin{align}
\nabla^{1,\alpha}_{e_2}e_1&=\nabla^{1,\alpha}_{(\overline{r_L}\overline{p}X_1+\overline{r_L}\overline{q}X_2-\frac{l}{l_L}\widetilde{X_3})}(\overline{q}X_1-\overline{p}X_2)\nonumber\\
&=[\overline{r_Lp}X_1(\overline{q})+\overline{r_Lq}X_2(\overline{q})-\frac{l}{l_L}\widetilde{X}_3(\overline{p})+\frac{\overline{p_L}\sqrt{L}}{2}]X_1\nonumber\\
&+[\frac{l}{l_L}\widetilde{X}_3(\overline{p})+\frac{\overline{q_L}\sqrt{L}}{2}-\overline{r_Lp}X_1(\overline{p})-\overline{r_Lq}X_2(\overline{p})]X_2\nonumber\\
&+[-\frac{(1-\alpha)\overline{r_L}\sqrt{L}}{2}+(1-\alpha)\overline{q_L}]\widetilde{X_3}.\nonumber\\
\end{align}
Then
\begin{align}
h_{21}&=-\langle\nabla^{1,\alpha}_{e_2}e_1,V_L\rangle_L\nonumber\\
&=-\frac{l}{l_L}[\overline{q}X_1(\overline{r_L})-\overline{p}X_2(\overline{r_L})]+\overline{r_L}[\overline{q}X_1(\frac{l}{l_L})-\overline{p}X_2(\frac{l}{l_L})]-\frac{(1+\alpha\overline{r_L}^2)\sqrt{L}}{2}+\alpha\overline{q_L}\overline{r_L},\nonumber\\
&=-\frac{l_L}{l}\langle e_1,\nabla_H(\overline{r_L})\rangle_L-\frac{(1+\alpha\overline{r_L}^2)\sqrt{L}}{2}+\alpha\overline{q_L}\overline{r_L}.\nonumber\\
\end{align}
Since
\begin{align}
\nabla^{1,\alpha}_{e_2}e_2&=\nabla^{1,\alpha}_{(\overline{r_L}\overline{p}X_1+\overline{r_L}\overline{q}X_2-\frac{l}{l_L}\widetilde{X_3})}(\overline{r_L}\overline{p}X_1+\overline{r_L}\overline{q}X_2-\frac{l}{l_L}\widetilde{X_3})\nonumber\\
&=[\overline{r_L}\overline{p}X_1(\overline{r_L}\overline{p})+\overline{r_L}\overline{q}X_2(\overline{r_L}\overline{p})-\frac{l}{l_L}\widetilde{X_3}(\overline{r_L}\overline{p})-\frac{(2-\alpha)\overline{r_L}\overline{q_L}\sqrt{L}}{2}+\frac{(1-\alpha)l^2}{l_L^2}]X_1\nonumber\\
&+[\overline{r_L}\overline{p}X_1(\overline{r_L}\overline{q})+\overline{r_L}\overline{q}X_2(\overline{r_L}\overline{q})-\frac{l}{l_L}\widetilde{X_3}(\overline{r_L}\overline{q})+\frac{(2-\alpha)\overline{r_L}\overline{p_L}\sqrt{L}}{2}]X_2\nonumber\\
&+[(1-\alpha)\overline{r_L}\overline{p}\frac{l}{l_L}-\overline{r_L}\overline{p}X_1(\frac{l}{l_L})-\overline{r_L}\overline{q}X_2(\frac{l}{l_L})+\frac{l}{l_L}\widetilde{X_3}(\frac{l}{l_L})]\widetilde{X_3}.\nonumber\\
\end{align}
Then,
\begin{align}
h_{22}&=-\langle\nabla^{1,\alpha}_{e_2}e_2,V_L\rangle_L\nonumber\\
&=-\langle\nabla^L_{e_2}e_2,V_L\rangle_L-(1-\alpha)\overline{p_L}\nonumber\\
&=-\frac{l^2}{l_L^2}\langle e_2,\nabla_H(\frac{r}{l})\rangle_L+\widetilde{X_3}(\overline{r_L})-(1-\alpha)\overline{p_L}.\nonumber\\
\end{align}
\end{proof}
\indent The Riemannian mean curvature $\mathcal{H}_{\nabla^{1,\alpha},L}$ of $\Sigma$ is defined by
$$\mathcal{H}_{\nabla^{1,\alpha},L}:={\rm tr}(II^{\nabla^{1,\alpha},L}).$$
Define the curvature of a connection $\nabla^{1,\alpha}$ by
\begin{equation}
R^{1,\alpha}(X,Y)Z=\nabla^{1,\alpha}_X\nabla^{1,\alpha}_YZ-\nabla^{1,\alpha}_Y\nabla^{1,\alpha}_XZ-\nabla^{1,\alpha}_{[X,Y]}Z.
\end{equation}
Let
\begin{equation}
\mathcal{K}^{\Sigma,\nabla^{1,\alpha}}(e_1,e_2)=-\langle R^{\Sigma,1,\alpha}(e_1,e_2)e_1,e_2\rangle_{\Sigma,L},~~~~\mathcal{K}^{\nabla^{1,\alpha}}(e_1,e_2)=-\langle R^{1,\alpha}(e_1,e_2)e_1,e_2\rangle_L.
\end{equation}
By the Gauss equation, we have
\begin{equation}
\mathcal{K}^{\Sigma,\nabla^{1,\alpha}}(e_1,e_2)=\mathcal{K}^{\nabla^{1,\alpha}}(e_1,e_2)+{\rm det}(II^{\nabla^{1,\alpha},L}).
\end{equation}
\vskip 0.5 true cm
\begin{prop} Away from characteristic points, the horizontal mean curvature associated to the first kind of deformed Schouten-Van Kampen connection $\nabla^{1,\alpha}$, $\mathcal{H}_{\nabla^{1,\alpha},\infty}$ of $\Sigma\subset\mathbb{M}$ is given by
\begin{equation}
\mathcal{H}_{\nabla^{1,\alpha},\infty}={\rm lim}_{L\rightarrow +\infty}\mathbb{}\mathcal{H}_{\nabla^{1,\alpha},L}=X_1(\overline{p})+X_2(\overline{q})-(1-\alpha)\overline{p}.
\end{equation}
\end{prop}
\begin{proof} By
$$\frac{l^2}{l_L^2}\langle e_2,\nabla_H(\frac{r}{l})\rangle_L=\frac{\overline{p}r}{l}X_1(\overline{r_L})+\frac{\overline{q}r}{l}X_2(\overline{r_L})=O(L^{-1})$$
$$\frac{l}{l_L}[X_1(\overline{p})+X_2(\overline{q})]\rightarrow X_1(\overline{p})+X_2(\overline{q}),~~~~\widetilde{X_3}(\overline{r_L})\rightarrow 0,$$
$$\overline{q}^2f_2\overline{r_L}L^{-\frac{1}{2}}\rightarrow O(L^{-1}),~~~~\overline{q_L}\rightarrow\overline{q},$$
$$\overline{r_L}\overline{P}^2f_2L^{-\frac{1}{2}}\rightarrow O(L^{-1}),~~~~\overline{p_L}\rightarrow\overline{p},$$
we get (2.57).
\end{proof}
By Lemma 2.1 and (2.54), we have
\begin{lem}
Let $\mathbb{M}$ be the affine group, then
\begin{align}
&
R^{1,\alpha}(X_1,X_2)X_1=\left[\frac{(1-\alpha)^2L}{4}+\frac{L}{2}\right]X_2+(1-\alpha)X_3,~~~ R^{1,\alpha}(X_1,X_2)X_2=-\left[\frac{(1-\alpha)L}{4}+\frac{L}{2}\right]X_1,\nonumber\\
&R^{1,\alpha}(X_1,X_2)X_3=-(1-\alpha)LX_1,~~~
R^{1,\alpha}(X_1,X_3)X_1=\frac{L}{2}\left[(1-\alpha)^2+1\right]X_2+\frac{(1-\alpha)(4-L)}{4}X_3,\nonumber\\ &R^{1,\alpha}(X_1,X_3)X_2=-\frac{(1-\alpha)^2+1}{2}LX_1,~~~
R^{1,\alpha}(X_1,X_3)X_3=\frac{(1-\alpha)(L^2-4L)}{4}X_1,\nonumber\\
&R^{1,\alpha}(X_2,X_3)X_1=0,~~~
R^{1,\alpha}(X_2,X_3)X_2=-\frac{(1-\alpha)L}{4}X_3,~~~ R^{1,\alpha}(X_2,X_3)X_3=\frac{(1-\alpha)L^2}{4}X_2.\nonumber\\
\end{align}
\end{lem}
\vskip 0.5 true cm
\begin{prop} Away from characteristic points, we have
\begin{equation}
\mathcal{K}^{\Sigma,\nabla^{1,\alpha}}(e_1,e_2)\rightarrow B_0+O(L^{-\frac{1}{2}}),~~{\rm as}~~L\rightarrow +\infty,
\end{equation}
where
\begin{align}
B_0&:=-\frac{(2-\alpha)}{2}\langle e_1,\nabla_H(\frac{X_3u}{|\nabla_Hu|})\rangle-(1-\alpha)\overline{q}^2+\frac{(4-3\alpha-\alpha^2)\overline{q}}{2}\frac{X_3u}{|\nabla_Hu|}-\frac{3-\alpha}{4}\left(\frac{X_3u}{|\nabla_Hu|}\right)^2\nonumber\\
&-(1-\alpha)\overline{p}\left[X_1(\overline{p})+X_2(\overline{q})\right].\nonumber\\
\end{align}
\end{prop}
\begin{proof} By (2.18), we have
\begin{align}
&~~~~\langle R^{1,\alpha}(e_1,e_2)e_1,e_2\rangle_L=\overline{r_L}^2\langle R^{1,\alpha}(X_1,X_2)X_1,X_2\rangle_L-2\frac{l}{l_L}\overline{q}L^{-\frac{1}{2}}\overline{r_L}\langle R^{1,\alpha}(X_1,X_2)X_1,X_3\rangle_L\nonumber\\
&+2\frac{l}{l_L}\overline{p}L^{-\frac{1}{2}}\overline{r_L}\langle R^{1,\alpha}(X_1,X_2)X_2,X_3\rangle_L+(\frac{l}{l_L}\overline{q})^2L^{-1}\langle R^{1,\alpha}(X_1,X_3)X_1,X_3\rangle_L\nonumber\\
&-2(\frac{l}{l_L})^2\overline{p}\overline{q}L^{-1}\langle R^{1,\alpha}(X_1,X_3)X_2,X_3\rangle_L+(\overline{p}\frac{l}{l_L})^2L^{-1}\langle R^{1,\alpha}(X_2,X_3)X_2,X_3\rangle_L.\nonumber\\
\end{align}
By Lemma 2.17, we have
\begin{align}
\mathcal{K}^{\nabla^{1,\alpha}}(e_1,e_2)&=\frac{(1-\alpha)(\overline{p_L}^2+\overline{q_L}^2)L}{4}-(1-\alpha)\overline{q_L}^2+2(1-\alpha)\overline{q_L}\overline{r_L}L^\frac{1}{2}-\frac{(1-\alpha)^2L+2L}{4}\overline{r_L}^2.\nonumber\\
\end{align}
By (2.45) and
$$\nabla_H(\overline{r_L})=L^{-\frac{1}{2}}\nabla_H(\frac{X_3u}{|\nabla_Hu|})+O(L^{-1})~~{\rm as}~~L\rightarrow +\infty,$$
we get
\begin{align}
{\rm det}(II^L)&=-\frac{(1-\alpha)L}{4}-\frac{2-\alpha}{2}\langle e_1,\nabla_H(\frac{X_3u}{|\nabla_Hu|})\rangle-(1-\alpha)\overline{p}\left[X_1(\overline{p})+X_2(\overline{q})\right]\nonumber\\
&+\frac{(\alpha-\alpha^2)\overline{q}}{2}\frac{X_3u}{|\nabla_Hu|}-\frac{\alpha-\alpha^2}{4}\left(\frac{X_3u}{|\nabla_Hu|}\right)^2+O(L^{-\frac{1}{2}}).\nonumber\\
\end{align}
By (2.56),(2,62),(2.63)
we get (2.59).
\end{proof}
Let us first consider the case of a regular curve $\gamma:[a,b]\rightarrow (\mathbb{M},g_L)$. We define the Riemannian length measure
$$ds_L=||\dot{\gamma}||_Ldt.$$
\begin{lem}
Let $\gamma:[a,b]\rightarrow (\mathbb{M},g_L)$ be a Euclidean $C^2$-smooth and regular curve. Let
\begin{equation}
d{s}:=|\omega(\dot{\gamma}(t))|dt,~~~~d\overline{s}:=\frac{1}{2}\frac{1}{|\omega(\dot{\gamma}(t))|}\left(\frac{\dot{\gamma}_1^2}{f^2}+\dot{\gamma}_3^2\right)dt.
\end{equation}
Then
\begin{equation}
{\rm lim}_{L\rightarrow +\infty}\frac{1}{\sqrt{L}}\int_{\gamma}ds_L=\int_a^bds.
\end{equation}
When $\omega(\dot{\gamma}(t))\neq 0$, we have
\begin{equation}
\frac{1}{\sqrt{L}}ds_L=ds+d\overline{s}L^{-1}+O(L^{-2}) ~~{\rm as}~~L\rightarrow +\infty.
\end{equation}
When $\omega(\dot{\gamma}(t))= 0$, we have
\begin{equation}
\frac{1}{\sqrt{L}}ds_L= \frac{1}{\sqrt{L}}\sqrt{\frac{\dot{\gamma}_1^2}{\gamma_1^2}+\dot{\gamma}_3^2}dt.
\end{equation}
\end{lem}
\begin{proof}
We know that $$||\dot{\gamma}(t)||_L=\sqrt{\left(\frac{\dot{\gamma}_1}{\gamma_1}\right)^2+\dot{\gamma}_3^2+L\omega(\dot{\gamma}(t))^2},$$ similar to the proof of Lemma 6.1 in \cite{BTV}, we can
prove (2.65).
When $\omega(\dot{\gamma}(t))\neq 0$, we have
$$
\frac{1}{\sqrt{L}}ds_L=\sqrt{L^{-1}\left(\left(\frac{\dot{\gamma}_1}{\gamma_1}\right)^2+\dot{\gamma}_3^2\right)+\omega(\dot{\gamma}(t))^2}dt.$$
Using the Taylor expansion, we can prove (2.66). From the definition of $ds_L$ and $\omega(\dot{\gamma}(t))= 0$, we get (2.67).
\end{proof}
Let $\Sigma\subset(\mathbb{M},g_L)$ be a Euclidean $C^2$-smooth surface and $\Sigma=\{u=0\}$. Let $d\sigma_{\Sigma,L}$ denote the surface measure on $\Sigma$ with respect to the Riemannian metric $g_L$. Then similai to Proposition $4.2$ in $[7]$, we have
\begin{align}
&{\rm lim}_{L\rightarrow +\infty}\frac{1}{\sqrt{L}}\int_{\Sigma}d\sigma_{\Sigma,L}=d\sigma_\Sigma:=(\overline{p}\omega_2-\overline{q}\omega_1)\wedge \omega.
\end{align}
We recall the local Gauss-Bonnet theorem for the metric connection (see Proposition 5.2 in \cite{BK}).
\begin{thm}
 Let $\Sigma$
  be an  an oriented compact two-dimensional manifold with many
boundary components $(\partial\Sigma)_i, i\in \{1,\cdot\cdot\cdot,n\},$ given by Euclidean $C^2$-smooth regular and
closed curves $\gamma_i :[0,2\pi]\rightarrow (\partial\Sigma)_i.$ Let $\nabla$ be a metric connection and $\mathcal{K}^\nabla$ be the Gauss
curvature associated to $\nabla$ and $k^{s,\nabla}_{\gamma_i}$
be the signed geodesic curvature associated to $\nabla$, then
\begin{equation}
\int_{\Sigma}\mathcal{K}^\nabla d\sigma_{\Sigma}+\sum_{i=1}^n\int_{\gamma_i}k^{s,\nabla}_{\gamma_i}d{s}=2\pi\chi(M).
\end{equation}
\end{thm}
By Lemma 2.14 and Proposition 2.18 and Theorem 2.20. Similar to the proof of Theorem $1.1$ in $[1]$, we have
\begin{thm}
 Let $\Sigma\subset (\mathbb{M},g_L)$
  be a regular surface with finitely many boundary components $(\partial\Sigma)_i,$ $i\in\{1,\cdots,n\}$, given by Euclidean $C^2$-smooth regular and closed curves $\gamma_i:[0,2\pi]\rightarrow (\partial\Sigma)_i$. Suppose that the characteristic set $C(\Sigma)$ satisfies $\mathcal{H}^1(C(\Sigma))=0$ where $\mathcal{H}^1(C(\Sigma))$ denotes the Euclidean $1$-dimensional Hausdorff measure of $C(\Sigma)$ and that
$||\nabla_Hu||_H^{-1}$ is locally summable with respect to the Euclidean $2$-dimensional Hausdorff measure
near the characteristic set $C(\Sigma)$, then
\begin{equation}
\int_{\Sigma}\mathcal{K}^{\Sigma,\nabla^{1,\alpha},\infty}d\sigma_{\Sigma}+\sum_{i=1}^n\int_{\gamma_i}k^{\infty,\nabla^{1,\alpha},s}_{\gamma_i,\Sigma}d{s}=0.
\end{equation}
\begin{proof}
By the Gauss-Bonnet theorem, we have
\begin{equation}
\int_{\Sigma}\mathcal{K}^{\Sigma,\nabla^{1,\alpha},L}\frac{1}{\sqrt{L}}d\sigma_{\Sigma,L}+\sum_{i=1}^n\int_{\gamma_i}k^{L,\nabla^{1,\alpha},s}_{\gamma_i,\Sigma}\frac{1}{\sqrt{L}}d{s}_L=2\pi\frac{\chi(\Sigma)}{\sqrt{L}}.
\end{equation}
So by (2.59),(2.66),(2.67),(2.70),(2.71), we get Theorem 2.21.
\end{proof}
\end{thm}

\section{Gauss-Bonnet theorems associated to the second kind of deformed Schouten-Van Kampen connection in the affine group}
\indent Let $H_2={\rm span}\{X_2,X_3\}$ be the second kind of horizontal distribution on $\mathbb{M}$, then $H_2^\bot={\rm span}\{X_1\}$. Let $\nabla$ be the Levi-Civita connection on $\mathbb{M}$ with respect to $g_L$, and we recall the Schouten-Van Kampen connection $\nabla^{2,\alpha,s}$ by the following formulas
\begin{equation}
\nabla^{2,\alpha,s}_XY=P^2\nabla_X{P^2Y}+P^{2,\bot}\nabla_X{P^{2,\bot} Y},
\end{equation}
where $P^2$(resp.$P^{2,\bot}$) be the projection on $H_2$ (resp.$H_2^\bot$).\\
Nextly, we define the second kind of deformed Schouten-Van Kampen connection which is a metric connection in the affine group:
\begin{align}
\nabla^{2,\alpha}_XY&=(1-\alpha)\nabla_XY+\alpha\nabla^{2,\alpha,s}_XY\nonumber\\
&=(1-\alpha)\nabla_XY+\alpha P^2\nabla_X{P^2Y}+\alpha P^{2,\bot}\nabla_X{P^{2,\bot} Y},\nonumber\\
\end{align}
where $\alpha$ is a constant.\\
 By lemma 2.1 in \cite{YS} and (3.2), we have the following lemma
\begin{lem}
Let $\mathbb{M}$ be the affine group, then
\begin{align}
&\nabla^{2,\alpha}_{X_1}X_1=0,~~~\nabla^{2,\alpha}_{X_1}X_2=\frac{1}{2}X_3,~~~ \nabla^{2,\alpha}_{X_1}X_3=-\frac{L}{2}X_2,\nonumber\\
&\nabla^{2,\alpha}_{X_2}X_1=-\frac{1-\alpha}{2}X_3,~~~\nabla^{2,\alpha}_{X_2}X_2=0,~~~\nabla^{2,\alpha}_{X_2}X_3=\frac{(1-\alpha)L}{2}X_1,\nonumber\\
&\nabla^{2,\alpha}_{X_3}X_1=-\frac{(1-\alpha)L}{2}X_2-(1-\alpha)X_3,~~~\nabla^{2,\alpha}_{X_3}X_2=\frac{(1-\alpha)L}{2}X_1,~~\nabla^{2,\alpha}_{X_3}X_3=(1-\alpha)LX_1.\nonumber\\
\end{align}
\end{lem}
Similar to lemma 2.4, we have
\begin{lem}
Let $\gamma:[a,b]\rightarrow (\mathbb{M},g_L)$ be a Euclidean $C^2$-smooth regular curve in the Riemannian manifold $(\mathbb{M},g_L)$. Then,
\begin{align}
k^{L,\nabla^{2,\alpha}}_{\gamma}&=\Bigg\{\Bigg\{\left[\frac{\ddot{\gamma}_1\gamma_1-(\dot{\gamma}_1)^2}{\gamma_1^2}+L\omega(\dot{\gamma}(t))\frac{(1-\alpha)\dot{\gamma_2}(t)}{\gamma_1}\right]^2+\left[\ddot{\gamma}_3-\frac{(2-\alpha)\dot{\gamma}_1L}{2\gamma_1}\omega(\dot{\gamma}(t)\right]^2\nonumber\\
&+L\left[\frac{d}{dt}\omega(\dot{\gamma}(t))+\frac{\alpha\dot{\gamma}_1\dot{\gamma}_3}{2\gamma_1}-\frac{(1-\alpha)\dot{\gamma}_1}{\gamma_1}\omega(\dot{\gamma}(t))\right]^2\Bigg\}
\cdot\left[\left(\frac{\dot{\gamma}_1}{\gamma_1}\right)^2+\dot{\gamma}_3^2+L(\omega(\dot{\gamma}(t)))^2\right]^{-2}\nonumber\\
&-\Bigg\{\frac{\dot{\gamma}_1}{\gamma_1}\left[\frac{\dot{\gamma}_1\ddot{\gamma}_1-(\dot{\gamma}_1)^2}{\gamma_1^2}+L\omega(\dot{\gamma}(t))\frac{(1-\alpha)\dot{\gamma_2}(t)}{\gamma_1}\right]+\dot{\gamma}_3(t)\left[\ddot{\gamma}_3-\frac{(2-\alpha)\dot{\gamma}_1L}{2\gamma_1}\omega(\dot{\gamma}(t)\right]\nonumber\\
&+L\omega(\dot{\gamma}(t))\left[\frac{d}{dt}\omega(\dot{\gamma}(t))+\frac{\alpha\dot{\gamma}_1\dot{\gamma}_3}{2\gamma_1}-\frac{(1-\alpha)\dot{\gamma}_1}{\gamma_1}\omega(\dot{\gamma}(t))\right]\Bigg\}^2\cdot\left[\left(\frac{\dot{\gamma}_1}{\gamma_1}\right)^2+\dot{\gamma}_3^2+L(\omega(\dot{\gamma}(t)))^2\right]^{-3}\Bigg\}^{\frac{1}{2}}.\nonumber\\
\end{align}
When $\omega(\dot{\gamma}(t))=0$, we have\\
\begin{align}
k^{L,\nabla^{2,\alpha}}_{\gamma}&=\Bigg\{\Bigg\{\left[\frac{\ddot{\gamma}_1\gamma_1-(\dot{\gamma}_1)^2}{\gamma_1^2}\right]^2+\ddot{\gamma}_3^2+L\left[\frac{d}{dt}\omega(\dot{\gamma}(t))+\frac{\alpha\dot{\gamma}_1\dot{\gamma}_3}{2\gamma_1}\right]^2\Bigg\}
\cdot\left[\left(\frac{\dot{\gamma}_1}{\gamma_1}\right)^2+\dot{\gamma}_3^2\right]^{-2}\nonumber\\
&-\Bigg\{\frac{\dot{\gamma}_1}{\gamma_1}\left[\frac{\dot{\gamma}_1\ddot{\gamma}_1-(\dot{\gamma}_1)^2}{\gamma_1^2}+\dot{\gamma}_3\ddot{\gamma}_3\right]\Bigg\}^2\cdot\left[\left(\frac{\dot{\gamma}_1}{\gamma_1}\right)^2+\dot{\gamma}_3^2\right]^{-3}\Bigg\}^{\frac{1}{2}}.\nonumber\\
\end{align}
\end{lem}
\begin{proof}
By (2.10) and (3.3), we have
\begin{align}
\nabla^{2,\alpha}_{\dot{\gamma}}\dot{\gamma}&=
\left[\frac{\ddot{\gamma}_1\gamma_1-(\dot{\gamma}_1)^2}{\gamma_1^2}+\frac{(1-\alpha)\dot{\gamma_2}(t)L}{\gamma_1}\omega(\dot{\gamma}(t))\right]X_1+\left[\ddot{\gamma}_3-\frac{(2-\alpha)\dot{\gamma}_1L}{2\gamma_1}\omega(\dot{\gamma}(t)\right]X_2\nonumber\\
&+\left[\frac{d}{dt}\omega(\dot{\gamma}(t))+\frac{\alpha\dot{\gamma}_1\dot{\gamma}_3}{2\gamma_1}-\frac{(1-\alpha)\dot{\gamma}_1}{\gamma_1}\omega(\dot{\gamma}(t))\right]X_3.\nonumber\\
\end{align}
By (2.7), (2.10) and (3.6), we get Lemma 3.2.
\end{proof}
\begin{lem}
Let $\gamma:[a,b]\rightarrow (\mathbb{M},g_L)$ be a Euclidean $C^2$-smooth regular curve in the Riemannian manifold $(\mathbb{M},g_L)$. Then\\
(1)when $\omega(\dot{\gamma}(t))\neq 0$,\\
\begin{equation}
k_{\gamma}^{\infty,\nabla^{2,\alpha}}=\frac{\sqrt{\left(\frac{(1-\alpha)\dot{\gamma}_2}{\gamma_1}\right)^2+\left(\frac{(2-\alpha)\dot{\gamma}_1}{2\gamma_1}\right)^2}}{|\omega(\dot{\gamma}(t))|},
\end{equation}
(2)when $\omega(\dot{\gamma}(t))= 0 ~~~and~~~\frac{d}{dt}(\omega(\dot{\gamma}(t)))+\frac{\alpha\dot{\gamma}_1\dot{\gamma}_3}{2\gamma_1}=0$,\\
\begin{align}
k^{\infty,\nabla^{2,\alpha}}_{\gamma}&=\Bigg\{\left\{\left[\frac{\ddot{\gamma}_1\gamma_1-\dot{\gamma}_1^2}{\gamma_1^2}\right]^2+\dot{\gamma}_3^2\right\}
\cdot\left[\left(\frac{\dot{\gamma}_1}{\gamma_1}\right)^2+\dot{\gamma}_3^2\right]^{-2}\nonumber\\
&-\left[\frac{\dot{\gamma}_1\ddot{\gamma}_1-\dot{\gamma}_1^3}{\gamma_1^3}+\dot{\gamma}_3\ddot{\gamma}_3\right]^2
\cdot\left[\left(\frac{\dot{\gamma}_1}{\gamma_1}\right)^2+\dot{\gamma}_3^2\right]^{-3}\Bigg\}^{\frac{1}{2}},\nonumber\\
\end{align}
(3)when $\omega(\dot{\gamma}(t))= 0 ~~~and~~~\frac{d}{dt}(\omega(\dot{\gamma}(t)))+\frac{\alpha\dot{\gamma}_1\dot{\gamma}_3}{2\gamma_1}\neq0$,\\
\begin{equation}
{\rm lim}_{L\rightarrow +\infty}\frac{k_{\gamma}^{L,\nabla^{2,\alpha}}}{\sqrt{L}}=\frac{|\frac{d}{dt}(\omega(\dot{\gamma}(t)))+\frac{\alpha\dot{\gamma}_1\dot{\gamma}_3}{2\gamma_1}|}{\left(\frac{\dot{\gamma}_1}{\gamma_1}\right)^2+\dot{\gamma}_3^2}.
\end{equation}
\end{lem}
For every $U,V\in T\Sigma$, we define $\nabla^{\Sigma,{2,\alpha}}_UV=\pi \nabla^{2,\alpha}_UV$ where $\pi:T\mathbb{M}\rightarrow T\Sigma$ is the projection. Then $\nabla^{\Sigma,{2,\alpha}}$ is the Levi-Civita connection on $\Sigma$
with respect to the metric $g_L$. By (2.18), (3.6) and
\begin{equation}
\nabla^{\Sigma,{2,\alpha}}_{\dot{\gamma}}\dot{\gamma}=\langle \nabla^{2,\alpha}_{\dot{\gamma}}\dot{\gamma},e_1\rangle_Le_1+\langle \nabla^{2,\alpha}_{\dot{\gamma}}\dot{\gamma},e_2\rangle_Le_2,
\end{equation}
we have
\begin{align}
\nabla^{\Sigma,{2,\alpha}}_{\dot{\gamma}}\dot{\gamma}&=
\Bigg\{\overline{q}\left[\frac{\ddot{\gamma}_1\gamma_1-(\dot{\gamma}_1)^2}{\gamma_1^2}+L\omega(\dot{\gamma}(t))\frac{(1-\alpha)\dot{\gamma_2}(t)}{\gamma_1}\right]-\overline{p}\left[\ddot{\gamma}_3-\frac{(2-\alpha)\dot{\gamma}_1L}{2\gamma_1}\omega(\dot{\gamma}(t)\right]\Bigg\}e_1\nonumber\\
&+\Bigg\{\overline{r_L}~~\overline{p}\left[\frac{\ddot{\gamma}_1\gamma_1-(\dot{\gamma}_1)^2}{\gamma_1^2}+L\omega(\dot{\gamma}(t))\frac{(1-\alpha)\dot{\gamma_2}(t)}{\gamma_1}\right]+\overline{r_L}~~ \overline{q}\left[\ddot{\gamma}_3-\frac{(2-\alpha)\dot{\gamma}_1L}{2\gamma_1}\omega(\dot{\gamma}(t)\right]\nonumber\\
&-\frac{l}{l_L}L^{\frac{1}{2}}\left[\frac{d}{dt}\omega(\dot{\gamma}(t))+\frac{\alpha\dot{\gamma}_1\dot{\gamma}_3}{2\gamma_1}-\frac{(1-\alpha)\dot{\gamma}_1}{\gamma_1}\omega(\dot{\gamma}(t))\right]\Bigg\}e_2.\nonumber\\
\end{align}
Moreover if $\omega(\dot{\gamma}(t))=0$, then
\begin{align}
\nabla^{\Sigma,{2,\alpha}}_{\dot{\gamma}}\dot{\gamma}&=
\Bigg\{\overline{q}\left[\frac{\ddot{\gamma}_1\gamma_1-(\dot{\gamma}_1)^2}{\gamma_1^2}\right]-\overline{p}\ddot{\gamma}_3\Bigg\}e_1\nonumber\\
&+\Bigg\{\overline{r_L}~~\overline{p}\frac{\ddot{\gamma}_1\gamma_1-(\dot{\gamma}_1)^2}{\gamma_1^2}+\overline{r_L}~~ \overline{q}
\ddot{\gamma}_3-\frac{l}{l_L}L^{\frac{1}{2}}\left[\frac{d}{dt}\omega(\dot{\gamma}(t))+\frac{\alpha\dot{\gamma}_1\dot{\gamma}_3}{2\gamma_1}\right]\Bigg\}e_2.\nonumber\\
\end{align}
\begin{lem}
Let $\Sigma\subset(\mathbb{M},g_L)$ be a regular surface.
Let $\gamma:[a,b]\rightarrow \Sigma$ be a Euclidean $C^2$-smooth regular curve. Then\\
(1)when $\omega(\dot{\gamma}(t))\neq 0,$
\begin{equation}
k_{\gamma,\Sigma}^{\infty,\nabla^{2,\alpha}}=\frac{|\overline{p}\frac{(2-\alpha)\dot{\gamma}_1}{2}+\overline{q}(1-\alpha)\dot{\gamma}_2|}{|\gamma_1\omega(\dot{\gamma}(t))|},
\end{equation}
(2)when $\omega(\dot{\gamma}(t))= 0, ~~~and~~~\frac{d}{dt}(\omega(\dot{\gamma}(t)))+\frac{\alpha\dot{\gamma}_1\dot{\gamma}_3}{2\gamma_1}=0,$
\begin{equation}
k_{\gamma,\Sigma}^{\infty,\nabla^{2,\alpha}}=0,\\
\end{equation}
(3)when $\omega(\dot{\gamma}(t))= 0, ~~~and~~~\frac{d}{dt}(\omega(\dot{\gamma}(t)))+\frac{\alpha\dot{\gamma}_1\dot{\gamma}_3}{2\gamma_1}\neq0,$
\begin{equation}
{\rm lim}_{L\rightarrow +\infty}\frac{k_{\gamma,\Sigma}^{L,\nabla^{2,\alpha}}}{\sqrt{L}}=\frac{|\frac{d}{dt}(\omega(\dot{\gamma}(t)))+\frac{\alpha\dot{\gamma}_1\dot{\gamma}_3}{2\gamma_1}|}{\left(\overline{q}\frac{\dot{\gamma}_1}{f}-\overline{p}\dot{\gamma}_3\right)^2}.
\end{equation}
\end{lem}
\begin{lem}
Let $\Sigma\subset(\mathbb{M},g_L)$ be a regular surface.
Let $\gamma:[a,b]\rightarrow \Sigma$ be a Euclidean $C^2$-smooth regular curve. Then\\
(1)when $\omega(\dot{\gamma}(t))\neq 0$,
\begin{equation}
k_{\gamma,\Sigma}^{\infty,\nabla^{2,\alpha},s}=\frac{|\overline{p}\frac{(2-\alpha)\dot{\gamma}_1}{2}+\overline{q}(1-\alpha)\dot{\gamma}_2|}{|\gamma_1\omega(\dot{\gamma}(t))|},
\end{equation}
(2)when $\omega(\dot{\gamma}(t))= 0, ~~ ~and~ ~~\frac{d}{dt}(\omega(\dot{\gamma}(t)))+\frac{\alpha\dot{\gamma}_1\dot{\gamma}_3}{2\gamma_1}=0$,
\begin{equation}
k^{\infty,\nabla^{2,\alpha},s}_{\gamma,\Sigma}=0,\\
\end{equation}
(3)when $\omega(\dot{\gamma}(t))= 0, ~~ ~and~ ~~\frac{d}{dt}(\omega(\dot{\gamma}(t)))+\frac{\alpha\dot{\gamma}_1\dot{\gamma}_3}{2\gamma_1}\neq0$,
\begin{equation}
{\rm lim}_{L\rightarrow +\infty}\frac{k_{\gamma,\Sigma}^{L,\nabla^{2,\alpha},s}}{\sqrt{L}}=\frac{(-\overline{q}\frac{\dot{\gamma}_1}{f}+\overline{p}\dot{\gamma}_3)\left[\frac{d}{dt}(\omega(\dot{\gamma}(t)))+\frac{\alpha\dot{\gamma}_1\dot{\gamma}_3}{2\gamma_1}\right]}{|\overline{q}\frac{\dot{\gamma}_1}{f}-\overline{p}\dot{\gamma}_3|^3}.
\end{equation}
\end{lem}
Similarly to Theorem 4.3 in \cite{CDPT}, we have
\begin{thm} The second fundamental form $II^{\nabla^{2,\alpha},L}$ of the
embedding of $\Sigma$ into $(\mathbb{M},g_L)$ is given by
\begin{equation}
II^{\nabla^{2,\alpha},L}=\left(
  \begin{array}{cc}
  h^{2,\alpha}_{11},
    & h^{2,\alpha}_{12} \\
  h^{2,\alpha}_{21} ,
    & h^{2,\alpha}_{22} \\
  \end{array}
\right),
\end{equation}
where
$$h^{2,\alpha}_{11}=\frac{l}{l_L}[X_1(\overline{p})+X_2(\overline{q})]+\frac{\alpha\overline{r_Lpq}\sqrt{L}}{2},$$
    $$h^{2,\alpha}_{12}=-\frac{l_L}{l}\langle e_1,\nabla_H(\overline{r_L})\rangle_L+\frac{(\alpha\overline{p_L}^2-1)\sqrt{L}}{2}+\frac{\alpha\overline{r_L}^2\overline{p}^2\sqrt{L}}{2},$$
    $$h^{2,\alpha}_{21}=-\frac{l_L}{l}\langle e_1,\nabla_H(\overline{r_L})\rangle_L-\frac{(1+\alpha\overline{r_L}\overline{q}^2-\alpha\overline{p_L}^2-\alpha\overline{q_L}^2)\sqrt{L}}{2}+\alpha\overline{q_L}\overline{r_L},$$
    $$h^{2,\alpha}_{22}=-\frac{l^2}{l_L^2}\langle e_2,\nabla_H(\frac{r}{l})\rangle_L+\widetilde{X_3}(\overline{r_L})-(1-\alpha)\overline{p_L}-\frac{\alpha\overline{r_L}(\overline{p_L}\overline{q_L}+\overline{r_L}^2\overline{p}\overline{q})\sqrt{L}}{2}.$$
\end{thm}
\indent The Riemannian mean curvature $\mathcal{H}_{\nabla^{2,\alpha},L}$ of $\Sigma$ is defined by
$$\mathcal{H}_{\nabla^{2,\alpha},L}:={\rm tr}(II^{\nabla^{2,\alpha},L}).$$
Then we have
\begin{prop} Away from characteristic points, the horizontal mean curvature associated to the second kind of deformed Schouten-Van Kampen connection $\nabla^{2,\alpha}$, $\mathcal{H}_{\nabla^{2,\alpha},\infty}$ of $\Sigma\subset\mathbb{M}$ is given by
\begin{equation}
\mathcal{H}_{\nabla^{2,\alpha},\infty}={\rm lim}_{L\rightarrow +\infty}\mathbb{}\mathcal{H}_{\nabla^{2,\alpha},L}=X_1(\overline{p})+X_2(\overline{q})-(1-\alpha)\overline{p}.
\end{equation}
\end{prop}
By Lemma 3.1 and (2.54), we have
\begin{lem}
Let $\mathbb{M}$ be the affine group, then
\begin{align}
&
R^{2,\alpha}(X_1,X_2)X_1=\frac{3(1-\alpha)L}{4}X_2+(1-\alpha)X_3,~~~ R^{2,\alpha}(X_1,X_2)X_2=-\frac{3(1-\alpha)L}{4}X_1,\nonumber\\
&R^{2,\alpha}(X_1,X_2)X_3=-(1-\alpha)LX_1,~~~
R^{2,\alpha}(X_1,X_3)X_1=(1-\alpha)LX_2+\frac{(1-\alpha)(4-L)}{4}X_3,\nonumber\\ &R^{2,\alpha}(X_1,X_3)X_2=-(1-\alpha)LX_1,~~~
R^{2,\alpha}(X_1,X_3)X_3=\frac{(1-\alpha)(L^2-4L)}{4}X_1,\nonumber\\
&R^{2,\alpha}(X_2,X_3)X_1=0,~~~
R^{2,\alpha}(X_2,X_3)X_2=-\frac{(1-\alpha)^2L}{4}X_3,~~~ R^{2,\alpha}(X_2,X_3)X_3=\frac{(1-\alpha)^2L^2}{4}X_2.\nonumber\\
\end{align}
\end{lem}
\begin{prop} Away from characteristic points, we have
\begin{equation}
\mathcal{K}^{\Sigma,\nabla^{2,\alpha}}(e_1,e_2)\rightarrow\frac{\alpha^2\overline{p}^2L}{2}+ B_1+O(L^{-1}),~~{\rm as}~~L\rightarrow +\infty,
\end{equation}
where
\begin{align}
B_1&:=-\frac{(2+\alpha\overline{q}^2)}{2}\langle e_1,\nabla_H(\frac{X_3u}{|\nabla_Hu|})\rangle-(1-\alpha)\overline{q}^2+\frac{\alpha\overline{p}^2(1+\alpha)+3\alpha-3}{4}\left(\frac{X_3u}{|\nabla_Hu|}\right)^2\nonumber\\
&+\frac{4\overline{q}^2-5\alpha\overline{q}+\alpha^2\overline{q}\overline{p}^2}{2}\frac{X_3u}{|\nabla_Hu|}-\left[(1-\alpha)\overline{p}+\frac{\alpha\overline{p}\overline{q}}{2}\frac{X_3u}{|\nabla_Hu|}\right]\cdot\left[X_1(\overline{p})+X_2(\overline{q})+\frac{\alpha\overline{p}\overline{q}}{2}\frac{X_3u}{|\nabla_Hu|}\right].
\end{align}
\end{prop}
\begin{thm}
 Let $\Sigma\subset (\mathbb{M},g_L)$
  be a regular surface with finitely many boundary components $(\partial\Sigma)_i,$ $i\in\{1,\cdots,n\}$, given by Euclidean $C^2$-smooth regular and closed curves $\gamma_i:[0,2\pi]\rightarrow (\partial\Sigma)_i$. Suppose that the characteristic set $C(\Sigma)$ satisfies $\mathcal{H}^1(C(\Sigma))=0$ where $\mathcal{H}^1(C(\Sigma))$ denotes the Euclidean $1$-dimensional Hausdorff measure of $C(\Sigma)$ and that
$||\nabla_Hu||_H^{-1}$ is locally summable with respect to the Euclidean $2$-dimensional Hausdorff measure
near the characteristic set $C(\Sigma)$, then
\begin{equation}
\int_{\Sigma}\frac{\alpha\overline{p}^2}{2}d\sigma_\Sigma=0,
\end{equation}
\begin{equation}
-\int_{\Sigma}\frac{\alpha\overline{p}^2}{2}d\overline{\sigma_\Sigma}+\int_{\Sigma}B_1d\sigma_\Sigma+\sum_{i=1}^n\int_{\gamma_i}k^{\infty,s}_{\gamma_i,\Sigma}d\overline{s}=0.
\end{equation}
\begin{proof}
Using the discussions in \cite{BTV1}, we know that the number of points satisfying
 $\omega(\dot{\gamma_i}(t))= 0$ and $\frac{d}{dt}(\omega(\dot{\gamma_i}(t)))\neq 0$ on $\gamma_i$ is finite. Since our proof of Theorem 3.10 is based on
 an approximation argument relying on the Lebesgue dominated convergence theorem. In the application of this theorem a set of finite many points can be ignored
 as a null set.
 Then by Lemma 3.5, we have
\begin{equation}
k^{L,\nabla^{2,\alpha},s}_{\gamma_i,\Sigma}=k^{\infty,\nabla^{2,\alpha},s}_{\gamma_i,\Sigma}+O(L^{-\frac{1}{2}}).
\end{equation}
We assume firstly that $C(\Sigma)$ is empty set.
By the Gauss-Bonnet theorem, we have
\begin{equation}
\int_{\Sigma}\mathcal{K}^{\Sigma,\nabla^{2,\alpha},L}\frac{1}{\sqrt{L}}d\sigma_{\Sigma,L}+\sum_{i=1}^n\int_{\gamma_i}k^{L,\nabla^{2,\alpha},s}_{\gamma_i,\Sigma}\frac{1}{\sqrt{L}}d{s}_L=2\pi\frac{\chi(\Sigma)}{\sqrt{L}}.
\end{equation}
So by (2.66), (2.67), (3.24) and (3.25), we get
\begin{equation}
-\left(\int_{\Sigma}\frac{\alpha\overline{p}^2}{2}d\sigma_\Sigma\right)L+\left(-\int_{\Sigma}\frac{\alpha\overline{p}^2}{2}d\overline{\sigma_\Sigma}+\int_{\Sigma}B_1d\sigma_\Sigma+\sum_{i=1}^n\int_{\gamma_i}k^{\infty,\nabla^{2,\alpha},s}_{\gamma_i,\Sigma}d\overline{s}\right)
+O(L^{-\frac{1}{2}})
=2\pi\frac{\chi(\Sigma)}{\sqrt{L}}.
\end{equation}
We multiply (3.28) by a factor $\frac{1}{L}$ and let $L$ go to the infinity and using the dominated convergence theorem, then we get (3.24). Using (3.24) and (3.28), we get (3.25).
Using the similar discussions of the page 27 in \cite{BTV}, we can relax the condition that the characteristic set $C(\Sigma)$ is the empty set and only suppose that the characteristic set $C(\Sigma)$ satisfies $\mathcal{H}^1(C(\Sigma))=0$
and that
$||\nabla_Hu||_H^{-1}$ is locally summable with respect to the Euclidean $2$-dimensional Hausdorff measure
near the characteristic set $C(\Sigma)$.\\
\end{proof}
\end{thm}
\section{Gauss-Bonnet theorems associated to the first kind of deformed Schouten-Van Kampen connection in the group of rigid motions of the Minkowski plane}
\indent We consider the group of rigid motions of the Minkowski plane $E(1,1)$, a unimodular Lie group with a natural subriemannian structure. As a model of $E(1,1),$ we choose the underlying manifold $\mathbb{R}^3$.
  On $\mathbb{R}^3$, we let
\begin{equation}
X_1=\partial_{x_3}, ~~X_2=\frac{1}{\sqrt{2}}(-e^{x_3}\partial_{x_1}+e^{-x_3}\partial_{x_2}),~~X_3=-\frac{1}{\sqrt{2}}(e^{x_3}\partial_{x_1}+e^{-x_3}\partial_{x_2}),
\end{equation}
with brackets
\begin{equation}
[X_1,X_2]=X_3,~~[X_1,X_3]=X_2,~~[X_2,X_3]=0.
\end{equation}
Then
\begin{equation}
\partial_{x_1}=-\frac{\sqrt{2}}{2}e^{-x_3}(X_2+X_3), ~~\partial_{x_2}=\frac{\sqrt{2}}{2}e^{x_3}(X_2-X_3),~~\partial_{x_3}=X_1,
\end{equation}
and ${\rm span}\{X_1,X_2,X_3\}=TE(1,1)$.
$\omega_1=dx_3,~~\omega_2=\frac{1}{\sqrt{2}}(-e^{-x_3}dx_1+e^{x_3}dx_2)
,~~\omega=-\frac{1}{\sqrt{2}}(e^{-x_3}dx_1+e^{x_3}dx_2)$. For the constant $L>0$, let
$g_L=\omega_1\otimes \omega_1+\omega_2\otimes \omega_2+L\omega\otimes \omega$ be the Riemannian metric on $E(1,1)$. Then $X_1,X_2,\widetilde{X_3}:=L^{-\frac{1}{2}}X_3$ are orthonormal basis on $TE(1,1)$ with respect to $g_L$.\\
\indent Let $H^1={\rm span}\{X_1,X_2\}$ be the first kind of horizontal distribution on $E(1,1)$, then $H^{1,\bot}={\rm span}\{X_3\}$.
Nextly, we define the first kind of deformed Schouten-Van Kampen connection which is a metric connection in the group of rigid motions of the Minkowski plane:
\begin{align}
\nabla^{1,\beta}_XY&=(1-\beta)\nabla_XY+\beta\nabla^{1,\beta,s}_XY\nonumber\\
&=(1-\beta)\nabla_XY+\beta P^1\nabla_X{P^1Y}+\beta P^{1,\bot}\nabla_X{P^{1,\bot} Y},\nonumber\\
\end{align}
where $\beta$ is a constant and $P^1$(resp.$P^{1,\bot}$) be the projection on $H^1$ (resp.$H^{1,\bot}$).\\
 By lemma 5.1 in \cite{YS} and (4.4), we have the following lemma
\begin{lem}
Let $E(1,1)$ be the group of rigid motions of the Minkowski plane, then
\begin{align}
&\nabla^{1,\beta}_{X_1}X_1=0,~~~\nabla^{1,\beta}_{X_1}X_2=\frac{(1-\beta)(L-1)}{2L}X_3,~~~ \nabla^{1,\beta}_{X_1}X_3=-\frac{(1-L)(1-\beta)}{2}X_2,\nonumber\\
&\nabla^{1,\beta}_{X_2}X_1=\frac{(1-\beta)(-L-1)}{2L}X_3,~~~\nabla^{1,\beta}_{X_2}X_2=0,~~~\nabla^{1,\beta}_{X_2}X_3=\frac{(1-\beta)(L+1)}{2}X_1,\nonumber\\
&\nabla^{1,\beta}_{X_3}X_1=-\frac{L+1}{2}X_2,~~~\nabla^{1,\beta}_{X_3}X_2=\frac{L+1}{2}X_1,~~\nabla^{1,\beta}_{X_3}X_3=0.\nonumber\\
\end{align}
\end{lem}
\begin{lem}
Let $\gamma:[a,b]\rightarrow (E(1,1),g_L)$ be a Euclidean $C^2$-smooth regular curve in the Riemannian manifold $(E(1,1),g_L)$. Then,
\begin{align}
k^{L,\nabla^{1,\beta}}_{\gamma}&=\Bigg\{\Bigg\{\left[\ddot{\gamma}_3+\frac{\sqrt{2}(L+1)}{4}\omega(\dot{\gamma}(t))(2-\beta)\left(-e^{-\gamma_3}\dot{\gamma_1}+e^{\gamma_3}\dot{\gamma_2}\right)\right]^2\nonumber\\
&+\left[\frac{\sqrt{2}}{2}(\ddot{\gamma}_2e^{\gamma_3}+\dot{\gamma}_2\dot{\gamma}_3e^{\gamma_3}-\ddot{\gamma}_1e^{-\gamma_3}+\dot{\gamma}_1\dot{\gamma}_3e^{-\gamma_3})+\frac{\beta L-\beta-2L}{2}\omega(\dot{\gamma}(t))\dot{\gamma}_3\right]^2\nonumber\\
&+L\left[\frac{d}{dt}(\omega(\dot{\gamma}(t)))-\frac{\sqrt{2}(1-\beta)}{2L}\left(-e^{-\gamma_3}\dot{\gamma_1}+e^{\gamma_3}\dot{\gamma_2}\right)\dot{\gamma}_3\right]^2\Bigg\}
\cdot\left[\dot{\gamma}_3^2+\frac{1}{2}\left(-e^{-\gamma_3}\dot{\gamma_1}+e^{\gamma_3}\dot{\gamma_2}\right)+L(\omega(\dot{\gamma}(t)))^2\right]^{-2}\nonumber\\
&-\Bigg\{\dot{\gamma}_3\left[\ddot{\gamma}_3+\frac{\sqrt{2}(L+1)}{4}\omega(\dot{\gamma}(t))(2-\beta)\left(-e^{-\gamma_3}\dot{\gamma_1}+e^{\gamma_3}\dot{\gamma_2}\right)\right]+\frac{\sqrt{2}}{2}\left(-e^{-\gamma_3}\dot{\gamma_1}+e^{\gamma_3}\dot{\gamma_2}\right)\nonumber\\
&\left[\frac{\sqrt{2}}{2}(\ddot{\gamma}_2e^{\gamma_3}+\dot{\gamma}_2\dot{\gamma}_3e^{\gamma_3}-\ddot{\gamma}_1e^{-\gamma_3}+\dot{\gamma}_1\dot{\gamma}_3e^{-\gamma_3})+\frac{\beta L-\beta-2L}{2}\omega(\dot{\gamma}(t))\dot{\gamma}_3\right]\nonumber\\
&+L\omega(\dot{\gamma}(t))\left[\frac{d}{dt}(\omega(\dot{\gamma}(t)))-\frac{\sqrt{2}(1-\beta)}{2L}\left(-e^{-\gamma_3}\dot{\gamma_1}+e^{\gamma_3}\dot{\gamma_2}\right)\dot{\gamma}_3\right]\Bigg\}^2\nonumber\\
&\cdot\left[\dot{\gamma}_3^2+\frac{1}{2}\left(-e^{-\gamma_3}\dot{\gamma_1}+e^{\gamma_3}\dot{\gamma_2}\right)+L(\omega(\dot{\gamma}(t)))^2\right]^{-3}\Bigg\}^{\frac{1}{2}}.\nonumber\\
\end{align}
When $\omega(\dot{\gamma}(t))=0$, we have\\
\begin{align}
k^{L,\nabla^{1,\beta}}_{\gamma}&=\Bigg\{\Bigg\{\ddot{\gamma}_3^2+\frac{1}{2}(\ddot{\gamma}_2e^{\gamma_3}+\dot{\gamma}_2\dot{\gamma}_3e^{\gamma_3}-\ddot{\gamma}_1e^{-\gamma_3}+\dot{\gamma}_1\dot{\gamma}_3e^{-\gamma_3})^2\nonumber\\
&+L\left[\frac{d}{dt}(\omega(\dot{\gamma}(t)))-\frac{\sqrt{2}(1-\beta)}{2L}\left(-e^{-\gamma_3}\dot{\gamma_1}+e^{\gamma_3}\dot{\gamma_2}\right)\dot{\gamma}_3\right]^2\Bigg\}\cdot\left[\dot{\gamma}_3^2+\frac{1}{2}\left(-e^{-\gamma_3}\dot{\gamma_1}+e^{\gamma_3}\dot{\gamma_2}\right)\right]^{-2}\nonumber\\
&-\Bigg\{\dot{\gamma}_3\ddot{\gamma}_3+\frac{1}{2}\left(-e^{-\gamma_3}\dot{\gamma_1}+e^{\gamma_3}\dot{\gamma_2}\right)(\ddot{\gamma}_2e^{\gamma_3}+\dot{\gamma}_2\dot{\gamma}_3e^{\gamma_3}-\ddot{\gamma}_1e^{-\gamma_3}+\dot{\gamma}_1\dot{\gamma}_3e^{-\gamma_3})\Bigg\}\nonumber\\
&\cdot\left[\dot{\gamma}_3^2+\frac{1}{2}\left(-e^{-\gamma_3}\dot{\gamma_1}+e^{\gamma_3}\dot{\gamma_2}\right)^2\right]^{-3}\Bigg\}^{\frac{1}{2}}.\nonumber\\
\end{align}
\end{lem}
\begin{proof}
By (4.3), we have
\begin{equation}
\dot{\gamma}(t)=\dot{\gamma}_3X_1+\frac{\sqrt{2}}{2}\left(-e^{-\gamma_3}\dot{\gamma_1}+e^{\gamma_3}\dot{\gamma_2}\right)X_2+\omega(\dot{\gamma}(t))X_3,
\end{equation}
where $\omega(\dot{\gamma}(t))=-\frac{\sqrt{2}}{2}\left(e^{-\gamma_3}\dot{\gamma_1}+e^{\gamma_3}\dot{\gamma_2}\right).$
By Lemma 4.1 and (4.8), we have
\begin{align}
&\nabla^{1,\beta}_{\dot{\gamma}}X_1=-\frac{L+1}{2}\omega(\dot{\gamma}(t))X_2-\frac{\sqrt{2}(L+1)(1-\beta)}{4L}\left(-e^{-\gamma_3}\dot{\gamma_1}+e^{\gamma_3}\dot{\gamma_2}\right)X_3
,\nonumber\\
&\nabla^{1,\beta}_{\dot{\gamma}}X_2=\frac{L+1}{2}\omega(\dot{\gamma}(t))X_1+\frac{(L-1)(1-\beta)}{2L}\dot{\gamma}_3X_3
,\nonumber\\
&\nabla^{1,\beta}_{\dot{\gamma}}X_3=\frac{\sqrt{2}(L+1)(1-\beta)}{4}\left(-e^{-\gamma_3}\dot{\gamma_1}+e^{\gamma_3}\dot{\gamma_2}\right)X_1+\frac{(1-L)(1-\beta)}{2}\dot{\gamma}_3X_2.
\nonumber\\
\end{align}
By (4.8) and (4.9), we have
\begin{align}
\nabla^{1,\beta}_{\dot{\gamma}}\dot{\gamma}&=
\left[\ddot{\gamma}_3+\frac{\sqrt{2}(L+1)(2-\beta)}{4}\left(-e^{-\gamma_3}\dot{\gamma_1}+e^{\gamma_3}\dot{\gamma_2}\right)\omega(\dot{\gamma}(t))\right]X_1\nonumber\\
&+\left[\frac{\sqrt{2}}{2}(\ddot{\gamma}_2e^{\gamma_3}+\dot{\gamma}_2\dot{\gamma}_3e^{\gamma_3}-\ddot{\gamma}_1e^{-\gamma_3}+\dot{\gamma}_1\dot{\gamma}_3e^{-\gamma_3})+\frac{\beta L-\beta-2L}{2}\omega(\dot{\gamma}(t))\dot{\gamma}_3\right]X_2
\nonumber\\
&+\left[\frac{d}{dt}(\omega(\dot{\gamma}(t)))-\frac{\sqrt{2}(1-\beta)}{2L}\left(-e^{-\gamma_3}\dot{\gamma_1}+e^{\gamma_3}\dot{\gamma_2}\right)\dot{\gamma}_3\right]X_3.\nonumber\\
\end{align}
By (4.8) and (4.10), we get Lemma 4.2.
\end{proof}
\begin{lem}
Let $\gamma:[a,b]\rightarrow (E(1,1),g_L)$ be a Euclidean $C^2$-smooth regular curve in the Riemannian manifold $(E(1,1),g_L)$. Then\\
(1)when $\omega(\dot{\gamma}(t))\neq 0$,\\
\begin{equation}
k_{\gamma}^{\infty,\nabla^{1,\beta}}=\frac{\sqrt{\frac{1}{8}\left[(2-\beta)\left(-e^{-\gamma_3}\dot{\gamma_1}+e^{\gamma_3}\dot{\gamma_2}\right)\right]^2+\frac{(2-\beta)^2\dot{\gamma}_3^2}{4}}}{|\omega(\dot{\gamma}(t))|},
\end{equation}
(2)when $\omega(\dot{\gamma}(t))= 0 ~~~and~~~\frac{d}{dt}(\omega(\dot{\gamma}(t)))=0$,\\
\begin{align}
k^{\infty,\nabla^{1,\beta}}_{\gamma}&=\Bigg\{\left\{\dot{\gamma}_3^2+\frac{1}{2}(\ddot{\gamma}_2e^{\gamma_3}+\dot{\gamma}_2\dot{\gamma}_3e^{\gamma_3}-\ddot{\gamma}_1e^{-\gamma_3}+\dot{\gamma}_1\dot{\gamma}_3e^{-\gamma_3})^2\right\}
\cdot\left[\dot{\gamma}_3^2+\frac{1}{2}\left(-e^{-\gamma_3}\dot{\gamma_1}+e^{\gamma_3}\dot{\gamma_2}\right)^2\right]^{-2}\nonumber\\
&-\Bigg\{\dot{\gamma_3}\ddot{\gamma}_3+\frac{1}{2}\left(-e^{-\gamma_3}\dot{\gamma_1}+e^{\gamma_3}\dot{\gamma_2}\right)\left(\ddot{\gamma}_2e^{\gamma_3}+\dot{\gamma}_2\dot{\gamma}_3e^{\gamma_3}-\ddot{\gamma}_1e^{-\gamma_3}+\dot{\gamma}_1\dot{\gamma}_3e^{-\gamma_3}\right)\Bigg\}^2\nonumber\\
&\cdot\left[\dot{\gamma}_3^2+\frac{1}{2}\left(-e^{-\gamma_3}\dot{\gamma_1}+e^{\gamma_3}\dot{\gamma_2}\right)^2\right]^{-3}\Bigg\}^{\frac{1}{2}},\nonumber\\
\end{align}
(3)when $\omega(\dot{\gamma}(t))= 0 ~~~and~~~\frac{d}{dt}(\omega(\dot{\gamma}(t)))\neq0$,\\
\begin{equation}
{\rm lim}_{L\rightarrow +\infty}\frac{k_{\gamma}^{L,\nabla^{1,\beta}}}{\sqrt{L}}=\frac{|\frac{d}{dt}(\omega(\dot{\gamma}(t)))|}{\dot{\gamma}_3^2+\frac{1}{2}\left(-e^{-\gamma_3}\dot{\gamma_1}+e^{\gamma_3}\dot{\gamma_2}\right)^2}.
\end{equation}
\end{lem}
For every $U,V\in T\Sigma$, we define $\nabla^{\Sigma,{1,\beta}}_UV=\pi \nabla^{1,\beta}_UV$ where $\pi:TE(1,1)\rightarrow T\Sigma$ is the projection. Then $\nabla^{\Sigma,{1,\beta}}$ is the Levi-Civita connection on $\Sigma$
with respect to the metric $g_L$. By (2.18),(4.10) and
\begin{equation}
\nabla^{\Sigma,{1,\beta}}_{\dot{\gamma}}\dot{\gamma}=\langle \nabla^{1,\beta}_{\dot{\gamma}}\dot{\gamma},e_1\rangle_Le_1+\langle \nabla^{1,\beta}_{\dot{\gamma}}\dot{\gamma},e_2\rangle_Le_2,
\end{equation}
we have
\begin{align}
\nabla^{\Sigma,{1,\beta}}_{\dot{\gamma}}\dot{\gamma}&=
\Bigg\{\overline{q}\left[\ddot{\gamma}_3+\frac{\sqrt{2}(L+1)(2-\beta)}{4}\left(-e^{-\gamma_3}\dot{\gamma_1}+e^{\gamma_3}\dot{\gamma_2}\right)\omega(\dot{\gamma}(t))\right]\nonumber\\
&-\overline{p}\left[\frac{\sqrt{2}}{2}(\ddot{\gamma}_2e^{\gamma_3}+\dot{\gamma}_2\dot{\gamma}_3e^{\gamma_3}-\ddot{\gamma}_1e^{-\gamma_3}+\dot{\gamma}_1\dot{\gamma}_3e^{-\gamma_3})+\frac{\beta L-\beta-2L}{2}\omega(\dot{\gamma}(t))\dot{\gamma}_3\right]\Bigg\}e_1\nonumber\\
&+\Bigg\{\overline{r_L}~~\overline{p}\left[\ddot{\gamma}_3+\frac{\sqrt{2}(L+1)(2-\beta)}{4}\left(-e^{-\gamma_3}\dot{\gamma_1}+e^{\gamma_3}\dot{\gamma_2}\right)\omega(\dot{\gamma}(t))\right]\nonumber\\
&+\overline{r_L}~~ \overline{q}\left[\frac{\sqrt{2}}{2}(\ddot{\gamma}_2e^{\gamma_3}+\dot{\gamma}_2\dot{\gamma}_3e^{\gamma_3}-\ddot{\gamma}_1e^{-\gamma_3}+\dot{\gamma}_1\dot{\gamma}_3e^{-\gamma_3})+\frac{\beta L-\beta-2L}{2}\omega(\dot{\gamma}(t))\dot{\gamma}_3\right]\nonumber\\
&-\frac{l}{l_L}L^{\frac{1}{2}}\left[\frac{d}{dt}(\omega(\dot{\gamma}(t)))-\frac{\sqrt{2}(1-\beta)}{2L}\left(-e^{-\gamma_3}\dot{\gamma_1}+e^{\gamma_3}\dot{\gamma_2}\right)\dot{\gamma}_3\right]\Bigg\}e_2.\nonumber\\
\end{align}
Moreover if $\omega(\dot{\gamma}(t))=0$, then
\begin{align}
\nabla^{\Sigma,{1,\beta}}_{\dot{\gamma}}\dot{\gamma}&=
\Bigg\{\overline{q}\dot{\gamma}_3-\frac{\sqrt{2}\overline{p}}{2}(\ddot{\gamma}_2e^{\gamma_3}+\dot{\gamma}_2\dot{\gamma}_3e^{\gamma_3}-\ddot{\gamma}_1e^{-\gamma_3}+\dot{\gamma}_1\dot{\gamma}_3e^{-\gamma_3})\Bigg\}e_1\nonumber\\
&+\Bigg\{\overline{r_L}~~\overline{p}\dot{\gamma}_3+\overline{r_L}~~ \overline{q}
\frac{\sqrt{2}}{2}(\ddot{\gamma}_2e^{\gamma_3}+\dot{\gamma}_2\dot{\gamma}_3e^{\gamma_3}-\ddot{\gamma}_1e^{-\gamma_3}+\dot{\gamma}_1\dot{\gamma}_3e^{-\gamma_3})\nonumber\\
&-\frac{l}{l_L}L^{\frac{1}{2}}\left[\frac{d}{dt}(\omega(\dot{\gamma}(t)))-\frac{\sqrt{2}(1-\beta)}{2L}\left(-e^{-\gamma_3}\dot{\gamma_1}+e^{\gamma_3}\dot{\gamma_2}\right)\dot{\gamma}_3\right]\Bigg\}e_2.\nonumber\\
\end{align}
\begin{lem}
Let $\Sigma\subset(E(1,1),g_L)$ be a regular surface and let $\gamma:[a,b]\rightarrow \Sigma$ be a Euclidean $C^2$-smooth regular curve. Then\\
(1)when $\omega(\dot{\gamma}(t))\neq 0,$
\begin{equation}
k_{\gamma,\Sigma}^{\infty,\nabla^{1,\beta}}=\frac{|\frac{\sqrt{2}(2-\beta)\overline{q}}{4}\left(-e^{-\gamma_3}\dot{\gamma_1}+e^{\gamma_3}\dot{\gamma_2}\right)+\frac{(2-\beta)\overline{p}\dot{\gamma_3}}{2}|}{|\omega(\dot{\gamma}(t))|},
\end{equation}
(2)when $\omega(\dot{\gamma}(t))= 0, ~~~and~~~\frac{d}{dt}(\omega(\dot{\gamma}(t)))=0,$
\begin{equation}
k_{\gamma,\Sigma}^{\infty,\nabla^{1,\beta}}=0,\\
\end{equation}
(3)when $\omega(\dot{\gamma}(t))= 0, ~~~and~~~\frac{d}{dt}(\omega(\dot{\gamma}(t)))\neq0,$
\begin{equation}
{\rm lim}_{L\rightarrow +\infty}\frac{k_{\gamma,\Sigma}^{L,\nabla^{1,\beta}}}{\sqrt{L}}=\frac{|\frac{d}{dt}(\omega(\dot{\gamma}(t)))|}{|\overline{q}\dot{\gamma}_3-\frac{\sqrt{2}\overline{p}}{2}\left(-e^{-\gamma_3}\dot{\gamma_1}+e^{\gamma_3}\dot{\gamma_2}\right)|}.
\end{equation}
\end{lem}
\begin{lem}
Let $\Sigma\subset(E(1,1),g_L)$ be a regular surface.
Let $\gamma:[a,b]\rightarrow \Sigma$ be a Euclidean $C^2$-smooth regular curve. Then\\
(1)when $\omega(\dot{\gamma}(t))\neq 0$,
\begin{equation}
k_{\gamma,\Sigma}^{\infty,\nabla^{1,\beta},s}=\frac{|\frac{\sqrt{2}(2-\beta)\overline{q}}{4}\left(-e^{-\gamma_3}\dot{\gamma_1}+e^{\gamma_3}\dot{\gamma_2}\right)+\frac{(2-\beta)\overline{p}\dot{\gamma_3}}{2}|}{|\omega(\dot{\gamma}(t))|},
\end{equation}
(2)when $\omega(\dot{\gamma}(t))= 0, ~~ ~and~ ~~\frac{d}{dt}(\omega(\dot{\gamma}(t)))=0$,
\begin{equation}
k^{\infty,\nabla^{1,\beta},s}_{\gamma,\Sigma}=0,\\
\end{equation}
(3)when $\omega(\dot{\gamma}(t))= 0, ~~ ~and~ ~~\frac{d}{dt}(\omega(\dot{\gamma}(t)))\neq0$,
\begin{equation}
{\rm lim}_{L\rightarrow +\infty}\frac{k_{\gamma,\Sigma}^{L,\nabla^{1,\beta},s}}{\sqrt{L}}=\frac{\left[-\overline{q}\dot{\gamma}_3+\frac{\sqrt{2}\overline{p}}{2}\left(-e^{-\gamma_3}\dot{\gamma_1}+e^{\gamma_3}\dot{\gamma_2}\right)\right]\frac{d}{dt}(\omega(\dot{\gamma}(t)))}{|\overline{q}\dot{\gamma}_3-\frac{\sqrt{2}\overline{p}}{2}\left(-e^{-\gamma_3}\dot{\gamma_1}+e^{\gamma_3}\dot{\gamma_2}\right)|^3}.
\end{equation}
\end{lem}
Similarly to Theorem 4.3 in \cite{CDPT}, we have
\vskip 0.5 true cm
\begin{thm} The second fundamental form $II^{\nabla^{1,\beta},L}$ of the
embedding of $\Sigma$ into $(E(1,1),g_L)$ is given by
\begin{equation}
II^{\nabla^{1,\beta},L}=\left(
  \begin{array}{cc}
  h^{1,\beta}_{11},
    & h^{1,\beta}_{12} \\
  h^{1,\beta}_{21} ,
    & h^{1,\beta}_{22} \\
  \end{array}
\right),
\end{equation}
where
$$h^{1,\beta}_{11}=\frac{l}{l_L}[X_1(\overline{p})+X_2(\overline{q})]-\overline{r_L}\overline{pq}(1-\beta)L^{-\frac{1}{2}},$$
    $$h^{1,\beta}_{12}=-\frac{l_L}{l}\langle e_1,\nabla_H(\overline{r_L})\rangle_L-\frac{(1-\beta)\sqrt{L}}{2}+\frac{(1-\beta)(\overline{q_L}^2-\overline{p_L}^2)}{2\sqrt{L}}+\frac{(1-\beta)\overline{r_L}^2(\overline{q}^2-\overline{p}^2)}{2\sqrt{L}},$$
    $$h^{1,\beta}_{21}=-\frac{l_L}{l}\langle e_1,\nabla_H(\overline{r_L})\rangle_L-\frac{\sqrt{L}}{2}+\frac{\overline{q_L}^2-\overline{p_L}^2}{2\sqrt{L}}+\frac{\overline{q}^2\overline{r_L}^2(1-\beta L-\beta)}{2\sqrt{L}}-\frac{\overline{p}^2\overline{r_L}^2(1+\beta L-\beta)}{2\sqrt{L}},$$
    $$h^{1,\beta}_{22}=-\frac{l^2}{l_L^2}\langle e_2,\nabla_H(\frac{r}{l})\rangle_L+\widetilde{X_3}(\overline{r_L})+\frac{(1-\alpha)\overline{p_L}\overline{q_L}\overline{r_L}}{\sqrt{L}}+\frac{(1-\alpha)\overline{p}\overline{q}\overline{r_L}^3}{\sqrt{L}}.$$
\end{thm}
\indent The Riemannian mean curvature $\mathcal{H}_{\nabla^{1,\beta},L}$ of $\Sigma$ is defined by
$$\mathcal{H}_{\nabla^{1,\beta},L}:={\rm tr}(II^{\nabla^{1,\beta},L}).$$
\begin{prop} Away from characteristic points, the horizontal mean curvature associated to the first kind of deformed Schouten-Van Kampen connection $\nabla^{1,\beta}$, $\mathcal{H}_{\nabla^{1,\beta},\infty}$ of $\Sigma\subset E(1,1)$ is given by
\begin{equation}
\mathcal{H}_{\nabla^{1,\beta},\infty}={\rm lim}_{L\rightarrow +\infty}\mathbb{}\mathcal{H}_{\nabla^{1,\beta},L}=X_1(\overline{p})+X_2(\overline{q}).
\end{equation}
\end{prop}
By Lemma 4.1 and (2.54), we have
\begin{lem}
Let $E(1,1)$ be the group of rigid motions of the Minkowski plane, then
\begin{align}
&R^{1,\beta}(X_1,X_2)X_1=\left[\frac{-(1-L^2)(1-\beta)^2}{4L}+\frac{1+L}{2}\right]X_2,~~~ R^{1,\beta}(X_1,X_2)X_2=\left[\frac{(1-L^2)(1-\beta)^2}{4L}-\frac{1+L}{2}\right]X_1,\nonumber\\
&R^{1,\beta}(X_1,X_2)X_3=0,~~~R^{1,\beta}(X_1,X_3)X_1=\frac{(-L^2+2L+3)(1-\beta)}{4L}X_3,~~~R^{1,\beta}(X_1,X_3)X_2=0,\nonumber\\
&R^{1,\beta}(X_1,X_3)X_3=\frac{(1-\beta)(L^2-L-2)}{2}X_1,~~~R^{1,\beta}(X_2,X_3)X_1=0,\nonumber\\
&R^{1,\beta}(X_2,X_3)X_2=-\frac{-(1-\beta)(L^2+2L+1)}{4L}X_3,~~~R^{1,\beta}(X_2,X_3)X_3=\frac{(1-\beta)(L^2+2L+1)}{4}X_2.\nonumber\\
\end{align}
\end{lem}
\vskip 0.5 true cm
\begin{prop} Away from characteristic points, we have
\begin{equation}
\mathcal{K}^{\Sigma,\nabla^{1,\beta}}(e_1,e_2)\rightarrow D_0+O(L^{-2}),~~{\rm as}~~L\rightarrow +\infty,
\end{equation}
where
\begin{align}
D_0&:=-\frac{(2-\beta)}{2}\langle e_1,\nabla_H(\frac{X_3u}{|\nabla_Hu|})\rangle-\frac{5(1-\beta)(\overline{q}^2-\overline{p}^2)}{4}+\left(\frac{(1-\beta)^2+2}{4}\frac{X_3u}{|\nabla_Hu|}\right)^2.
\end{align}
\end{prop}
\begin{thm}
 Let $\Sigma\subset (E(1,1),g_L)$
  be a regular surface with finitely many boundary components $(\partial\Sigma)_i,$ $i\in\{1,\cdots,n\}$, given by Euclidean $C^2$-smooth regular and closed curves $\gamma_i:[0,2\pi]\rightarrow (\partial\Sigma)_i$. Suppose that the characteristic set $C(\Sigma)$ satisfies $\mathcal{H}^1(C(\Sigma))=0$ where $\mathcal{H}^1(C(\Sigma))$ denotes the Euclidean $1$-dimensional Hausdorff measure of $C(\Sigma)$ and that
$||\nabla_Hu||_H^{-1}$ is locally summable with respect to the Euclidean $2$-dimensional Hausdorff measure
near the characteristic set $C(\Sigma)$, then
\begin{equation}
\int_{\Sigma}\mathcal{K}^{\Sigma,\nabla^{1,\beta},\infty}d\sigma_{\Sigma}+\sum_{i=1}^n\int_{\gamma_i}k^{\infty,\nabla^{1,\beta},s}_{\gamma_i,\Sigma}d{s}=0.
\end{equation}
\end{thm}
\section{Gauss-Bonnet theorems associated to the second kind of deformed Schouten-Van Kampen connection in the group of rigid motions of the Minkowski plane}
\indent Let $H^2={\rm span}\{X_2,X_3\}$ be the second kind of horizontal distribution on $E(1,1)$, then $H^{2,\bot}={\rm span}\{X_1\}$.\\
Nextly, we define the second kind of deformed Schouten-Van Kampen connection which is a metric connection in the group of rigid motions of the Minkowski plane:
\begin{align}
\nabla^{2,\beta}_XY&=(1-\beta)\nabla_XY+\beta\nabla^{2,\beta,s}_XY\nonumber\\
&=(1-\beta)\nabla_XY+\beta P^2\nabla_X{P^2Y}+\beta P^{2,\bot}\nabla_X{P^{2,\bot} Y},\nonumber\\
\end{align}
where $\beta$ is a constant and $P^2$(resp.$P^{2,\bot}$) be the projection on $H^2$ (resp.$H^{2,\bot}$).\\
 By lemma 5.1 in \cite{YS} and (5.1), we have the following lemma
\begin{lem}
Let $E(1,1)$ be the group of rigid motions of the Minkowski plane, then
\begin{align}
&\nabla^{2,\beta}_{X_1}X_1=0,~~~\nabla^{2,\beta}_{X_1}X_2=\frac{L-1}{2L}X_3,~~~ \nabla^{2,\beta}_{X_1}X_3=\frac{1-L}{2}X_2,\nonumber\\
&\nabla^{2,\beta}_{X_2}X_1=-\frac{(-L-1)(1-\beta)}{2L}X_3,~~~\nabla^{2,\beta}_{X_2}X_2=0,~~~\nabla^{2,\beta}_{X_2}X_3=\frac{(1-\beta)(L+1)}{2}X_1,\nonumber\\
&\nabla^{2,\beta}_{X_3}X_1=-\frac{(-1-L)(1-\beta)}{2}X_2,~~~\nabla^{2,\beta}_{X_3}X_2=\frac{(1-\beta)(L+1)}{2}X_1,~~\nabla^{2,\beta}_{X_3}X_3=0.\nonumber\\
\end{align}
\end{lem}
Then, we have
\begin{lem}
Let $\gamma:[a,b]\rightarrow (E(1,1),g_L)$ be a Euclidean $C^2$-smooth regular curve in the Riemannian manifold $(E(1,1),g_L)$. Then,
\begin{align}
k^{L,\nabla^{2,\beta}}_{\gamma}&=\Bigg\{\Bigg\{\left[\ddot{\gamma}_3+\frac{\sqrt{2}(L+1)(1-\beta)}{2}\omega(\dot{\gamma}(t))\left(-e^{-\gamma_3}\dot{\gamma_1}+e^{\gamma_3}\dot{\gamma_2}\right)\right]^2\nonumber\\
&+\left[\frac{\sqrt{2}}{2}(\ddot{\gamma}_2e^{\gamma_3}+\dot{\gamma}_2\dot{\gamma}_3e^{\gamma_3}-\ddot{\gamma}_1e^{-\gamma_3}+\dot{\gamma}_1\dot{\gamma}_3e^{-\gamma_3})+\frac{\beta L+\beta-2L}{2}\omega(\dot{\gamma}(t))\dot{\gamma}_3\right]^2\nonumber\\
&+L\left[\frac{d}{dt}(\omega(\dot{\gamma}(t)))-\frac{\sqrt{2}}{2L}\left(-e^{-\gamma_3}\dot{\gamma_1}+e^{\gamma_3}\dot{\gamma_2}\right)\dot{\gamma}_3\right]^2\Bigg\}
\cdot\left[\dot{\gamma}_3^2+\frac{1}{2}\left(-e^{-\gamma_3}\dot{\gamma_1}+e^{\gamma_3}\dot{\gamma_2}\right)+L(\omega(\dot{\gamma}(t)))^2\right]^{-2}\nonumber\\
&-\Bigg\{\dot{\gamma}_3\left[\ddot{\gamma}_3+\frac{\sqrt{2}(L+1)(1-\beta)}{2}\omega(\dot{\gamma}(t))\left(-e^{-\gamma_3}\dot{\gamma_1}+e^{\gamma_3}\dot{\gamma_2}\right)\right]+\frac{\sqrt{2}}{2}\left(-e^{-\gamma_3}\dot{\gamma_1}+e^{\gamma_3}\dot{\gamma_2}\right)\nonumber\\
&\left[\frac{\sqrt{2}}{2}(\ddot{\gamma}_2e^{\gamma_3}+\dot{\gamma}_2\dot{\gamma}_3e^{\gamma_3}-\ddot{\gamma}_1e^{-\gamma_3}+\dot{\gamma}_1\dot{\gamma}_3e^{-\gamma_3})+\frac{\beta L+\beta-2L}{2}\omega(\dot{\gamma}(t))\dot{\gamma}_3\right]\nonumber\\
&+L\omega(\dot{\gamma}(t))\left[\frac{d}{dt}(\omega(\dot{\gamma}(t)))-\frac{\sqrt{2}}{2L}\left(-e^{-\gamma_3}\dot{\gamma_1}+e^{\gamma_3}\dot{\gamma_2}\right)\dot{\gamma}_3\right]\Bigg\}^2\nonumber\\
&\cdot\left[\dot{\gamma}_3^2+\frac{1}{2}\left(-e^{-\gamma_3}\dot{\gamma_1}+e^{\gamma_3}\dot{\gamma_2}\right)+L(\omega(\dot{\gamma}(t)))^2\right]^{-3}\Bigg\}^{\frac{1}{2}}.\nonumber\\
\end{align}
When $\omega(\dot{\gamma}(t))=0$, we have\\
\begin{align}
k^{L,\nabla^{2,\beta}}_{\gamma}&=\Bigg\{\Bigg\{\ddot{\gamma}_3^2+\frac{1}{2}(\ddot{\gamma}_2e^{\gamma_3}+\dot{\gamma}_2\dot{\gamma}_3e^{\gamma_3}-\ddot{\gamma}_1e^{-\gamma_3}+\dot{\gamma}_1\dot{\gamma}_3e^{-\gamma_3})^2\nonumber\\
&+L\left[\frac{d}{dt}(\omega(\dot{\gamma}(t)))-\frac{\sqrt{2}}{2L}\left(-e^{-\gamma_3}\dot{\gamma_1}+e^{\gamma_3}\dot{\gamma_2}\right)\dot{\gamma}_3\right]^2\Bigg\}\cdot\left[\dot{\gamma}_3^2\frac{1}{2}\left(-e^{-\gamma_3}\dot{\gamma_1}+e^{\gamma_3}\dot{\gamma_2}\right)\right]^{-2}\nonumber\\
&-\Bigg\{\dot{\gamma}_3\ddot{\gamma}_3+\frac{1}{2}\left(-e^{-\gamma_3}\dot{\gamma_1}+e^{\gamma_3}\dot{\gamma_2}\right)(\ddot{\gamma}_2e^{\gamma_3}+\dot{\gamma}_2\dot{\gamma}_3e^{\gamma_3}-\ddot{\gamma}_1e^{-\gamma_3}+\dot{\gamma}_1\dot{\gamma}_3e^{-\gamma_3})\Bigg\}\nonumber\\
&\cdot\left[\dot{\gamma}_3^2+\frac{1}{2}\left(-e^{-\gamma_3}\dot{\gamma_1}+e^{\gamma_3}\dot{\gamma_2}\right)^2\right]^{-3}\Bigg\}^{\frac{1}{2}}.\nonumber\\
\end{align}
\end{lem}
\begin{lem}
Let $\gamma:[a,b]\rightarrow (E(1,1),g_L)$ be a Euclidean $C^2$-smooth regular curve in the Riemannian manifold $(E(1,1),g_L)$. Then\\
(1)when $\omega(\dot{\gamma}(t))\neq 0$,\\
\begin{equation}
k_{\gamma}^{\infty,\nabla^{2,\beta}}=\frac{\sqrt{\left[(1-\beta)\left(-e^{-\gamma_3}\dot{\gamma_1}+e^{\gamma_3}\dot{\gamma_2}\right)\right]^2+(2-\beta)^2\dot{\gamma}_3^2}}{|2\omega(\dot{\gamma}(t))|},
\end{equation}
(2)when $\omega(\dot{\gamma}(t))= 0 ~~~and~~~\frac{d}{dt}(\omega(\dot{\gamma}(t)))=0$,\\
\begin{align}
k^{\infty,\nabla^{2,\beta}}_{\gamma}&=\Bigg\{\left\{\dot{\gamma}_3^2+\frac{1}{2}(\ddot{\gamma}_2e^{\gamma_3}+\dot{\gamma}_2\dot{\gamma}_3e^{\gamma_3}-\ddot{\gamma}_1e^{-\gamma_3}+\dot{\gamma}_1\dot{\gamma}_3e^{-\gamma_3})^2\right\}
\cdot\left[\dot{\gamma}_3^2+\frac{1}{2}\left(-e^{-\gamma_3}\dot{\gamma_1}+e^{\gamma_3}\dot{\gamma_2}\right)^2\right]^{-2}\nonumber\\
&-\Bigg\{\dot{\gamma_3}\ddot{\gamma}_3+\frac{1}{2}\left(-e^{-\gamma_3}\dot{\gamma_1}+e^{\gamma_3}\dot{\gamma_2}\right)\left(\ddot{\gamma}_2e^{\gamma_3}+\dot{\gamma}_2\dot{\gamma}_3e^{\gamma_3}-\ddot{\gamma}_1e^{-\gamma_3}+\dot{\gamma}_1\dot{\gamma}_3e^{-\gamma_3}\right)\Bigg\}^2\nonumber\\
&\cdot\left[\dot{\gamma}_3^2+\frac{1}{2}\left(-e^{-\gamma_3}\dot{\gamma_1}+e^{\gamma_3}\dot{\gamma_2}\right)^2\right]^{-3}\Bigg\}^{\frac{1}{2}},\nonumber\\
\end{align}
(3)when $\omega(\dot{\gamma}(t))= 0 ~~~and~~~\frac{d}{dt}(\omega(\dot{\gamma}(t)))\neq0$,\\
\begin{equation}
{\rm lim}_{L\rightarrow +\infty}\frac{k_{\gamma}^{L,\nabla^{2,\beta}}}{\sqrt{L}}=\frac{|\frac{d}{dt}(\omega(\dot{\gamma}(t)))|}{\dot{\gamma}_3^2+\frac{1}{2}\left(-e^{-\gamma_3}\dot{\gamma_1}+e^{\gamma_3}\dot{\gamma_2}\right)^2}.
\end{equation}
\end{lem}
For every $U,V\in T\Sigma$, we define $\nabla^{\Sigma,{2,\beta}}_UV=\pi \nabla^{2,\beta}_UV$ where $\pi:TE(1,1)\rightarrow T\Sigma$ is the projection. Then $\nabla^{\Sigma,{2,\beta}}$ is the Levi-Civita connection on $\Sigma$
with respect to the metric $g_L$. By (2.18) and
\begin{equation}
\nabla^{\Sigma,{2,\beta}}_{\dot{\gamma}}\dot{\gamma}=\langle \nabla^{2,\beta}_{\dot{\gamma}}\dot{\gamma},e_1\rangle_Le_1+\langle \nabla^{2,\beta}_{\dot{\gamma}}\dot{\gamma},e_2\rangle_Le_2,
\end{equation}
we have
\begin{align}
\nabla^{\Sigma,{2,\beta}}_{\dot{\gamma}}\dot{\gamma}&=
\Bigg\{\overline{q}\left[\ddot{\gamma}_3+\frac{\sqrt{2}(L+1)(1-\beta)}{2}\left(-e^{-\gamma_3}\dot{\gamma_1}+e^{\gamma_3}\dot{\gamma_2}\right)\omega(\dot{\gamma}(t))\right]\nonumber\\
&-\overline{p}\left[\frac{\sqrt{2}}{2}(\ddot{\gamma}_2e^{\gamma_3}+\dot{\gamma}_2\dot{\gamma}_3e^{\gamma_3}-\ddot{\gamma}_1e^{-\gamma_3}+\dot{\gamma}_1\dot{\gamma}_3e^{-\gamma_3})+\frac{\beta L+\beta-2L}{2}\omega(\dot{\gamma}(t))\dot{\gamma}_3\right]\Bigg\}e_1\nonumber\\
&+\Bigg\{\overline{r_L}~~\overline{p}\left[\ddot{\gamma}_3+\frac{\sqrt{2}(L+1)(1-\beta)}{2}\left(-e^{-\gamma_3}\dot{\gamma_1}+e^{\gamma_3}\dot{\gamma_2}\right)\omega(\dot{\gamma}(t))\right]\nonumber\\
&+\overline{r_L}~~ \overline{q}\left[\frac{\sqrt{2}}{2}(\ddot{\gamma}_2e^{\gamma_3}+\dot{\gamma}_2\dot{\gamma}_3e^{\gamma_3}-\ddot{\gamma}_1e^{-\gamma_3}+\dot{\gamma}_1\dot{\gamma}_3e^{-\gamma_3})+\frac{\beta L+\beta-2L}{2}\omega(\dot{\gamma}(t))\dot{\gamma}_3\right]\nonumber\\
&-\frac{l}{l_L}L^{\frac{1}{2}}\left[\frac{d}{dt}(\omega(\dot{\gamma}(t)))-\frac{\sqrt{2}}{2L}\left(-e^{-\gamma_3}\dot{\gamma_1}+e^{\gamma_3}\dot{\gamma_2}\right)\dot{\gamma}_3\right]\Bigg\}e_2.\nonumber\\
\end{align}
Moreover if $\omega(\dot{\gamma}(t))=0$, then
\begin{align}
\nabla^{\Sigma,\nabla^{2,\beta}}_{\dot{\gamma}}\dot{\gamma}&=
\Bigg\{\overline{q}\dot{\gamma}_3-\frac{\sqrt{2}\overline{p}}{2}(\ddot{\gamma}_2e^{\gamma_3}+\dot{\gamma}_2\dot{\gamma}_3e^{\gamma_3}-\ddot{\gamma}_1e^{-\gamma_3}+\dot{\gamma}_1\dot{\gamma}_3e^{-\gamma_3})\Bigg\}e_1\nonumber\\
&+\Bigg\{\overline{r_L}~~\overline{p}\dot{\gamma}_3+\overline{r_L}~~ \overline{q}
\frac{\sqrt{2}}{2}(\ddot{\gamma}_2e^{\gamma_3}+\dot{\gamma}_2\dot{\gamma}_3e^{\gamma_3}-\ddot{\gamma}_1e^{-\gamma_3}+\dot{\gamma}_1\dot{\gamma}_3e^{-\gamma_3})\nonumber\\
&-\frac{l}{l_L}L^{\frac{1}{2}}\left[\frac{d}{dt}(\omega(\dot{\gamma}(t)))-\frac{\sqrt{2}}{2L}\left(-e^{-\gamma_3}\dot{\gamma_1}+e^{\gamma_3}\dot{\gamma_2}\right)\dot{\gamma}_3\right]\Bigg\}e_2.\nonumber\\
\end{align}
\begin{lem}
Let $\Sigma\subset(E(1,1),g_L)$ be a regular surface.
Let $\gamma:[a,b]\rightarrow \Sigma$ be a Euclidean $C^2$-smooth regular curve. Then\\
(1)when $\omega(\dot{\gamma}(t))\neq 0,$
\begin{equation}
k_{\gamma,\Sigma}^{\infty,\nabla^{2,\beta}}=\frac{|\sqrt{2}(1-\beta)\overline{q}\left(-e^{-\gamma_3}\dot{\gamma_1}+e^{\gamma_3}\dot{\gamma_2}\right)+(2-\beta)\overline{p}\dot{\gamma_3}|}{|2\omega(\dot{\gamma}(t))|},
\end{equation}
(2)when $\omega(\dot{\gamma}(t))= 0, ~~~and~~~\frac{d}{dt}(\omega(\dot{\gamma}(t)))=0,$
\begin{equation}
k_{\gamma,\Sigma}^{\infty,\nabla^{2,\beta}}=0,\\
\end{equation}
(3)when $\omega(\dot{\gamma}(t))= 0, ~~~and~~~\frac{d}{dt}(\omega(\dot{\gamma}(t)))\neq0,$
\begin{equation}
{\rm lim}_{L\rightarrow +\infty}\frac{k_{\gamma,\Sigma}^{L,\nabla^{2,\beta}}}{\sqrt{L}}=\frac{|\frac{d}{dt}(\omega(\dot{\gamma}(t)))|}{|\overline{q}\dot{\gamma}_3-\frac{\sqrt{2}\overline{p}}{2}\left(-e^{-\gamma_3}\dot{\gamma_1}+e^{\gamma_3}\dot{\gamma_2}\right)|}.
\end{equation}
\end{lem}
\begin{lem}
Let $\Sigma\subset(E(1,1),g_L)$ be a regular surface.
Let $\gamma:[a,b]\rightarrow \Sigma$ be a Euclidean $C^2$-smooth regular curve. Then\\
(1)when $\omega(\dot{\gamma}(t))\neq 0$,
\begin{equation}
k_{\gamma,\Sigma}^{\infty,\nabla^{2,\beta},s}=\frac{|\sqrt{2}(1-\beta)\overline{q}\left(-e^{-\gamma_3}\dot{\gamma_1}+e^{\gamma_3}\dot{\gamma_2}\right)+(2-\beta)\overline{p}\dot{\gamma_3}|}{|2\omega(\dot{\gamma}(t))|},
\end{equation}
(2)when $\omega(\dot{\gamma}(t))= 0, ~~ ~and~ ~~\frac{d}{dt}(\omega(\dot{\gamma}(t)))=0$,
\begin{equation}
k^{\infty,\nabla^{2,\beta},s}_{\gamma,\Sigma}=0,\\
\end{equation}
(3)when $\omega(\dot{\gamma}(t))= 0, ~~ ~and~ ~~\frac{d}{dt}(\omega(\dot{\gamma}(t)))\neq0$,
\begin{equation}
{\rm lim}_{L\rightarrow +\infty}\frac{k_{\gamma,\Sigma}^{L,\nabla^{2,\beta},s}}{\sqrt{L}}=\frac{\left[-\overline{q}\dot{\gamma}_3+\frac{\sqrt{2}\overline{p}}{2}\left(-e^{-\gamma_3}\dot{\gamma_1}+e^{\gamma_3}\dot{\gamma_2}\right)\right]\frac{d}{dt}(\omega(\dot{\gamma}(t)))}{|\overline{q}\dot{\gamma}_3-\frac{\sqrt{2}\overline{p}}{2}\left(-e^{-\gamma_3}\dot{\gamma_1}+e^{\gamma_3}\dot{\gamma_2}\right)|^3}.
\end{equation}
\end{lem}
Similarly to Theorem 4.3 in \cite{CDPT}, we have
\vskip 0.5 true cm
\begin{thm} The second fundamental form $II^{\nabla^{2,\beta},L}$ of the
embedding of $\Sigma$ into $(E(1,1),g_L)$ is given by
\begin{equation}
II^{\nabla^{2,\beta},L}=\left(
  \begin{array}{cc}
  h^{2,\beta}_{11},
    & h^{2,\beta}_{12} \\
  h^{2,\beta}_{21} ,
    & h^{2,\beta}_{22} \\
  \end{array}
\right),
\end{equation}
where
$$h^{2,\beta}_{11}=\frac{l}{l_L}[X_1(\overline{p})+X_2(\overline{q})]+\frac{\beta L+\beta-2}{2}\overline{r_L}\overline{pq}L^{-\frac{1}{2}},$$
    $$h^{2,\beta}_{12}=-\frac{l_L}{l}\langle e_1,\nabla_H(\overline{r_L})\rangle_L-\frac{(1-\beta\overline{p_L}^2)\sqrt{L}}{2}+\frac{\overline{q_L}^2-\overline{p_L}^2(1-\beta)}{2\sqrt{L}}+\frac{[(1-\beta)\overline{p}^2-\overline{q}^2]\overline{r_L}^2}{2\sqrt{L}},$$
    $$h^{2,\beta}_{21}=-\frac{l_L}{l}\langle e_1,\nabla_H(\overline{r_L})\rangle_L+\frac{(\beta\frac{l}{l_L}\overline{p_L}-\beta\frac{l}{l_L}\overline{q_L}-1)\sqrt{L}}{2}-\frac{\overline{q_L}\beta\frac{l}{l_L}+\overline{p_L}\beta\frac{l}{l_L}}{2\sqrt{L}}+\frac{\overline{q}^2\overline{r_L}^2(-L-1 )}{2\sqrt{L}},$$
    $$h^{2,\beta}_{22}=-\frac{l^2}{l_L^2}\langle e_2,\nabla_H(\frac{r}{l})\rangle_L+\widetilde{X_3}(\overline{r_L})-\beta\overline{p_L}\overline{q_L}\overline{r_L}L^{\frac{1}{2}}+\frac{(1-\alpha)\overline{p_L}\overline{q_L}\overline{r_L}}{\sqrt{L}}+\frac{(1-\beta)(-L-1)\overline{p_L}\overline{q_L}\overline{r_L}}{2\sqrt{L}}-\frac{(\beta L +\beta-2)\overline{p}\overline{q}\overline{r_L}^2}{2}.$$
\end{thm}
\indent The Riemannian mean curvature $\mathcal{H}_{\nabla^{2,\beta},L}$ of $\Sigma$ is defined by
$$\mathcal{H}_{\nabla^{2,\beta},L}:={\rm tr}(II^{\nabla^{2,\beta},L}).$$
\begin{prop} Away from characteristic points, the horizontal mean curvature associated to the second kind of deformed canonical connection $\nabla^{2,\beta}$, $\mathcal{H}_{\nabla^{2,\beta},\infty}$ of $\Sigma\subset E(1,1)$ is given by
\begin{equation}
\mathcal{H}_{\nabla^{2,\beta},\infty}={\rm lim}_{L\rightarrow +\infty}\mathbb{}\mathcal{H}_{\nabla^{2,\beta},L}=X_1(\overline{p})+X_2(\overline{q})-\frac{\overline{pq}}{2}\frac{X_3u}{|\nabla_Hu|}-\frac{\beta\overline{pq}}{2}\left(\frac{X_3u}{|\nabla_Hu|}\right)^2.
\end{equation}
\end{prop}
By Lemma 5.1 and (2.54), we have
\begin{lem}
Let $E(1,1)$ be the group of rigid motions of the Minkowski plane, then
\begin{align}
&R^{2,\beta}(X_1,X_2)X_1=\frac{(2L^2+L-1)(1-\beta)}{2L}X_2,~~~ R^{2,\beta}(X_1,X_2)X_2=-\frac{L(L+1)(1-\beta)}{2}X_1,\nonumber\\
&R^{2,\beta}(X_1,X_2)X_3=0,~~~R^{2,\beta}(X_1,X_3)X_1=\frac{(-L^2+2L+3)(1-\beta)}{4L}X_3,~~~R^{2,\beta}(X_1,X_3)X_2=0,\nonumber\\
&R^{2,\beta}(X_1,X_3)X_3=\frac{(1-\beta)(L^2-2L-3)}{4}X_1,~~~R^{2,\beta}(X_2,X_3)X_1=0,\nonumber\\
&R^{2,\beta}(X_2,X_3)X_2=-\frac{-(1-\beta)^2(L^2+2L+1)}{4L}X_3,~~~R^{2,\beta}(X_2,X_3)X_3=\frac{(1-\beta)^2(L^2+2L+1)}{4}X_2.\nonumber\\
\end{align}
\end{lem}
\begin{prop} Away from characteristic points, we have
\begin{equation}
\mathcal{K}^{\Sigma,\nabla^{2,\beta}}(e_1,e_2)\rightarrow \frac{[\beta(\beta-2-\beta\overline{p}^3+\beta\overline{p}^2\overline{q}+\overline{p}^2+\overline{p}-\overline{q})]L}{4}+D_1+O(L^{-2}),~~{\rm as}~~L\rightarrow +\infty,
\end{equation}
where
\begin{align}
D_1&:=\frac{\beta\overline{p}^2+\beta\overline{p}+\beta\overline{q}-2}{2}\langle e_1,\nabla_H(\frac{X_3u}{|\nabla_Hu|})\rangle+\frac{(1-\beta)^2(\overline{p}^2-\overline{q}^2)}{2}+(1-\beta)\left(\frac{X_3u}{|\nabla_Hu|}\right)^2\nonumber\\
&+\left[\frac{-(\beta+1)\overline{pq}}{2}\frac{X_3u}{|\nabla_Hu|}-\frac{\beta\overline{pq}}{2}\left(\frac{X_3u}{|\nabla_Hu|}\right)^2\right]\cdot\left[X_1(\overline{p})+X_2(\overline{q})+\frac{\beta\overline{pq}}{2}\frac{X_3u}{|\nabla_Hu|}\right]\nonumber\\
&-\frac{(\beta\overline{p}-\beta\overline{q}-1)[\overline{q}^2-\overline{p}^2(1-\beta)]-\beta(\overline{p}+\overline{q})(\beta\overline{p}^2-1)}{4}.\nonumber\\
\end{align}
\end{prop}
\begin{thm}
 Let $\Sigma\subset (E(1,1),g_L)$
  be a regular surface with finitely many boundary components $(\partial\Sigma)_i,$ $i\in\{1,\cdots,n\}$, given by Euclidean $C^2$-smooth regular and closed curves $\gamma_i:[0,2\pi]\rightarrow (\partial\Sigma)_i$. Suppose that the characteristic set $C(\Sigma)$ satisfies $\mathcal{H}^1(C(\Sigma))=0$ where $\mathcal{H}^1(C(\Sigma))$ denotes the Euclidean $1$-dimensional Hausdorff measure of $C(\Sigma)$ and that
$||\nabla_Hu||_H^{-1}$ is locally summable with respect to the Euclidean $2$-dimensional Hausdorff measure
near the characteristic set $C(\Sigma)$, then
\begin{equation}
\int_{\Sigma}\frac{\beta(\beta-2-\beta\overline{p}^3+\beta\overline{p}^2\overline{q}+\overline{p}^2+\overline{p}-\overline{q})}{4}d\sigma_\Sigma=0,
\end{equation}
\begin{equation}
-\int_{\Sigma}\frac{\beta(\beta-2-\beta\overline{p}^3+\beta\overline{p}^2\overline{q}+\overline{p}^2+\overline{p}-\overline{q})}{4}d\overline{\sigma_\Sigma}+\int_{\Sigma}D_1d\sigma_\Sigma+\sum_{i=1}^n\int_{\gamma_i}k^{\infty,s}_{\gamma_i,\Sigma}d\overline{s}=0.
\end{equation}
\end{thm}
\section*{Acknowledgements}
The author was supported in part by  NSFC No.11771070. The author thanks the referee for his (or her) careful reading and helpful comments.

\end{document}